\documentclass[11pt,a4paper]{article}
\usepackage[utf8]{inputenc}
\usepackage[displaymath, mathlines]{lineno}
\usepackage{tikz}
\usetikzlibrary{arrows}
\usetikzlibrary{shadows} 
\RequirePackage[OT1]{fontenc}
\RequirePackage{amsthm,amsmath}
\RequirePackage[numbers]{natbib}
\usepackage{bbm, amssymb, kbordermatrix, graphicx, subcaption, changepage, nicefrac}
\usepackage{amsthm}
\usepackage{enumerate}
\usepackage[left=1in,right=1in,top=1in,bottom=1in]{geometry}
\usepackage[ruled,noend,nofillcomment]{algorithm2e}
\usepackage{wrapfig}
\usepackage{enumitem}
\usepackage{setspace}
\usepackage{times} 
\usepackage{helvet}
\usepackage{courier}
\usepackage{indentfirst}

\allowdisplaybreaks

\DeclareMathOperator{\diam}{Diam}

\newcommand{\J}{\mathcal{J}}
\newcommand{\Jmin}{\mathcal{J}_{\tiny\mbox{min}}}
\newcommand{\X}{\mathcal{X}}
\newcommand{\Y}{\mathcal{Y}}
\newcommand{\calU}{\mathcal{U}}
\newcommand{\calV}{\mathcal{V}}

\newcommand{\XN}{\mathcal{X}^{\mathbb{N}}}
\newcommand{\YN}{\mathcal{Y}^{\mathbb{N}}}

\newcommand{\calO}{\mathcal{O}}

\newcommand{\calM}{\mathcal{M}}

\newcommand{\cX}{c_{\scriptscriptstyle{\X}}}
\newcommand{\cY}{c_{\scriptscriptstyle{\Y}}}
\newcommand{\cU}{c_{\scriptscriptstyle{\calU}}}
\newcommand{\cV}{c_{\scriptscriptstyle{\calV}}}
\newcommand{\cXk}{c_{\scriptscriptstyle{\X} \!, k}}
\newcommand{\cYk}{c_{\scriptscriptstyle{\Y} \!, k}}

\newcommand{\ckXk}{(c_k)_{\scriptscriptstyle{\X^k}}}
\newcommand{\ckYk}{(c_k)_{\scriptscriptstyle{\Y^k}}}

\newcommand{\ot}{T}
\newcommand{\otc}{\ot_c}
\newcommand{\otck}{\ot_{c_k}}
\newcommand{\otcbar}{\ot_{\overline{c}}}
\newcommand{\eot}{\ot^\eta}

\newcommand{\eotc}{\ot^\eta_c}
\newcommand{\eotck}{\ot^\eta_{c_k}}

\newcommand{\otcXk}{\ot_{\cXk}}

\newcommand{\otcYk}{\ot_{\cYk}}
\newcommand{\otcU}{\ot_{\cU}}

\newcommand{\oj}{S}
\newcommand{\ojc}{\oj_c}
\newcommand{\ojcbar}{\oj_{\overline{c}}}
\newcommand{\eoj}{\oj^\eta}
\newcommand{\eojc}{\oj^\eta_c}

\newcommand{\ojhat}{\hat{\oj}}
\newcommand{\ojhatk}{\ojhat_k}
\newcommand{\ojhatkn}{\ojhat_{k,n}}
\newcommand{\ojhatknn}{\ojhat_{k(n),n}}
\newcommand{\estoj}{\hat{\lambda}}

\newcommand{\estojkn}{\estoj^{k,n}}
\newcommand{\estojknn}{\estoj^{k(n),n}}
\newcommand{\eojhat}{\tilde{\oj}}
\newcommand{\eojhatk}{\eojhat_k}
\newcommand{\eojhatkn}{\eojhat_{k,n}}
\newcommand{\eojhatetak}{\eojhatk^\eta}
\newcommand{\eojhatetakn}{\eojhatkn^\eta}
\newcommand{\esteoj}{\tilde{\lambda}}

\newcommand{\esteojkn}{\esteoj^{k,n}}

\newcommand{\esteojetakn}{\esteoj^{\eta, k, n}}

\newcommand{\tf}{\tilde{f}}
\newcommand{\tg}{\tilde{g}}

\newcommand{\tfb}{\tilde{f}_\beta}
\newcommand{\tga}{\tilde{g}_\alpha}

\newcommand{\tgap}{\tilde{g}_{\alpha'}}
\newcommand{\tU}{\tilde{U}}
\newcommand{\tV}{\tilde{V}}

\newcommand{\bfu}{\mathbf{u}}

\newcommand{\bfy}{\mathbf{y}}
\newcommand{\bfx}{\mathbf{x}}

\newcommand{\bbE}{\mathbb{E}}
\newcommand{\bbN}{\mathbb{N}}

\newcommand{\bbR}{\mathbb{R}}
\newcommand{\1}{\mathbbm{1}}

\newcommand{\eqd}{\stackrel{d}{=}}
\newcommand{\timesdots}{\! \times \! \cdots \! \times \!}

\usepackage{thm-restate}

\newtheorem{thm}{Theorem}
\newtheorem*{thm*}{Theorem}
\newtheorem{prop}[thm]{Proposition}

\newtheorem{lem}[thm]{Lemma}

\newtheorem{ethm}{Theorem}

\newtheorem{elem}[ethm]{Lemma}

\newtheorem{defn}{Definition}
\theoremstyle{remark}
\newtheorem{rem}{Remark}
\newtheorem{ex}{Example} 

\theoremstyle{plain}

\numberwithin{equation}{section}

\begin{document}
\title{\textbf{Estimation of Stationary Optimal Transport Plans}}
\author{Kevin O'Connor \\ UNC-Chapel Hill \and Kevin McGoff \\ UNC-Charlotte \and Andrew B Nobel \\ UNC-Chapel Hill}
\maketitle

\begin{abstract}
We study optimal transport for stationary stochastic processes taking values in finite spaces.
In order to reflect the stationarity of the underlying processes, we restrict attention to stationary couplings, also known as joinings.
The resulting optimal joining problem captures differences in the long run average behavior of the processes of interest.
We introduce estimators of both optimal joinings and the optimal joining cost, and we establish consistency of the estimators under mild conditions.  
Furthermore, under stronger mixing assumptions we establish finite-sample error rates for the estimated optimal joining cost that extend the best known results in the iid case.
Finally, we extend the consistency and rate analysis to an entropy-penalized version of the optimal joining problem.
\end{abstract}


\section{Introduction}
\label{sec:introduction}
The application and theory of optimal transport has recently received a great deal of attention in statistics and machine 
learning.  This work has resulted in
novel approaches to statistical estimation \citep{chen2020wasserstein, janati2019wasserstein, frogner2015learning, bassetti2006minimum, bernton2019parameter}, 
deep generative modeling \citep{arjovsky2017wasserstein, tolstikhin2018wasserstein}, 
clustering \citep{ho2017multilevel, laclau2017co, mi2018variational}, cell modeling \citep{schiebinger2019optimal}, and other applications \citep{courty2016optimal, courty2016optimalb, flamary2016optimal, evans2012phylogenetic, wang2010optimal}.  
In most of this work the objects under study (for example, images, text documents, graphs, and point clouds) 
are regarded as static and do not evolve over time.  
Accordingly, the distributions and cost functions appearing in the optimal transport problem capture the behavior of these objects at a fixed
point in time.  
In this static setting the statistical properties of the optimal transport problem, such as definition of estimators, consistency, and rates of convergence, have been well-studied \citep{panaretos2019statistical, peyre2019computational, tameling2019empirical, sommerfeld2016inference, rippl2016limit, hutter2021minimax, deb2021rates, perrot2016mapping, seguy2017large}.

In contrast with the static situation, we are interested in optimal transport problems in settings where the objects of interest are processes 
that evolve dynamically over time.
Examples include the alignment or generative modeling of text sequences or musical scores, or hybrid settings 
involving dynamic text and images.
Other examples include transportation of goods between a set of manufacturers and a set of retailers when
supply and demand vary over time in a stochastic fashion, or the comparison of brain networks observed at multiple points in time.
Problems involving sequential data or objects varying over time are pervasive in statistics and machine learning.
While standard optimal transport techniques can be applied to such problems, they
do not account for the structure of the underlying measures, which reflect dynamic processes rather than static quantities. 

In this paper we investigate optimal transport for finite alphabet stationary ergodic processes, together with cost functions that measure differences at a single time point (or finitely many time points).
Optimal transport for stationary processes is a special case of the ordinary optimal transport problem 
in which the distributions of interest are shift invariant measures on infinite product spaces (the sequence spaces associated with the given processes).  
As such, existing methods and theory apply. 
However, it is easy to show that a coupling of two stationary processes need not be stationary, and the same is true of optimal transport plans (see Examples \ref{ex:nonstationary_coupling} and \ref{ex:ot_is_bad} below).  
To address these, and other, issues arising in the general setting, we restrict our attention to {\em stationary} couplings of stationary processes.  
This seemingly mild restriction has far reaching consequences.

In restricting to stationary couplings, also known as {\em joinings}, we draw connections between optimal transport and ergodic theory.
In particular, properties and theoretical applications of joinings have been studied in ergodic theory for many years.
Given two stationary processes, the {\em optimal joining problem} is the problem of finding a joining
of the processes having minimal expected cost. 
The primary focus of this paper is estimating an optimal joining, and the associated optimal joining cost, 
of two finite alphabet stationary ergodic processes
using $n$ observations from each process.  
Roughly speaking, we use the available observations to estimate the $k$-dimensional distribution of each process, find an optimal coupling of these $k$-dimensional estimates, and then use this coupling to construct a joint process that is stationary.
In order to ensure that the constructed process converges to an optimal joining, it is necessary to balance estimating the $k$-dimensional distribution by letting the sample size $n$ grow and learning the dependence structure of the optimal joining by letting $k$ grow.
Thus the task of choosing an appropriate sequence $\{k(n)\}$ of block sizes indexed by sample size is critical for consistently estimating an optimal joining.

Under the stated assumptions, we show that there exists a sequence $\{k(n)\}$ for which the corresponding joint processes will converge to an optimal joining of the marginal processes, and the expected cost of the joint processes will converge to the cost of the optimal joining (Theorem \ref{thm:consistent_estimation_of_oj}).
This result generalizes existing results regarding the estimation of optimal joining costs and, to our knowledge, provides the first consistency result for estimating optimal joinings.
Under additional mixing assumptions on the observed processes, we identify an explicit growth rate for $k$ and obtain rates of convergence for estimates of the optimal joining cost (Theorem \ref{thm:abstract_rates}).   
To the best of our knowledge, these are the first finite-sample bounds for estimation of an optimal joining cost.
In the iid case, optimal joining and optimal transport coincide, and we recover existing, state-of-the-art bounds for estimation of the optimal transport cost.  
As special cases of our results we obtain new, finite-sample bounds for estimation of the $\overline{d}$- and $\overline{\rho}$-distances between stationary ergodic processes.

In recent years, there has been a substantial amount of work on regularized optimal transport, in which the regularization is obtained by adding an entropic penalty to the usual optimal transport cost.
As with (unregularized) optimal transport, the regularized problem has been primarily studied in static settings.
In this work, we bring these ideas to a dynamic setting.
More specifically, we propose and analyze a regularized form of the optimal joining 
problem, which is obtained by adding a penalty based on the entropy rate of a process to the expected cost.
We then extend our estimation scheme from the standard optimal joining problem to the regularized problem, 
and we establish both consistency and rate results for the resulting estimates (Theorems \ref{thm:consistent_estimation_of_eoj} and \ref{thm:eoj_error_bound}).
Existing algorithms for computing regularized optimal transport plans may be applied to compute the proposed estimates in the regularized case more efficiently compared to the unregularized case.


\paragraph{Organization of the Paper.}
The rest of the paper is organized as follows.
Background on optimal transport, optimal joinings, some initial results, and related work are presented in the next section.
In Section \ref{sec:estimate_oj}, we detail the proposed estimation scheme for an optimal 
joining and its expected cost and state our main consistency result.
Statement of our finite sample error bound under mixing assumptions, and a corollary, are presented in Section \ref{sec:rates}.
In Section \ref{sec:entropic_oj}, we introduce the entropic optimal joining problem and discuss how our estimation scheme, consistency result, and error bound extend to this problem.
We close with a discussion of our results in Section \ref{sec:discussion}.
Proofs of the main results are presented in Section \ref{sec:proofs}.

\section{Preliminaries and First Results}\label{sec:preliminaries}

We begin by defining the optimal transport problem.
Let $\calU$ and $\calV$ be metric spaces and let $\calM(\calU)$ and $\calM(\calV)$ denote the set of Borel probability measures on these spaces.
Recall that a Borel probability measure $\pi \in \calM(\calU \times \calV)$ on the product space $\calU \times \calV$ 
is said to be a {\em coupling} of $\mu \in \calM(\calU)$ and $\nu \in \calM(\calV)$ if for every measurable $A \subset \calU$ and $B \subset \calV$, $\pi(A \times \calV) = \mu(A)$ and $\pi(\calU \times B) = \nu(B)$, that is, 
the $\calU$-marginal of $\pi$ is $\mu$ and the $\calV$-marginal of $\pi$ is $\nu$.  
Given a non-negative cost function $c: \calU \times \calV \rightarrow \bbR_+$, the optimal transport problem is to minimize the expectation of $c$ over the set of couplings $\Pi(\mu, \nu)$:
\begin{equation}\label{eq:ot_problem}
\otc(\mu, \nu) 
=
\inf\limits_{\pi \in \Pi(\mu, \nu)} \int c \, d\pi.
\end{equation}
Any optimal solution to \eqref{eq:ot_problem} is referred to as an optimal transport plan (or coupling), and the quantity $\otc(\mu, \nu)$ is referred to as the optimal transport cost of $\mu$ and $\nu$ with respect to $c$.
For an introduction to couplings and the optimal transport problem, we refer the reader to \cite{villani2008optimal}.
 
As noted in Section \ref{sec:introduction}, the optimal transport problem is general, and does not incorporate 
information about about the structure of the sets $\calU$ and $\calV$, or the measures $\mu$ and $\nu$.
Existing work has considered different choices of $\calU$ and $\calV$, including finite dimensional Euclidean spaces \citep{tolstikhin2018wasserstein, arjovsky2017wasserstein}, graphs \citep{xu2019scalable, xu2019gromov, titouan2019optimal}, trees \citep{wang2020optimal}, finite sets \citep{sommerfeld2016inference, montrucchio2019kantorovich}, and sequence spaces \citep{o2020optimal, kolesnikov2017optimal}.
In what follows, we consider the latter case in which $\mu$ and $\nu$ are shift-invariant probability measures on sequence spaces, equivalently
the distributions of stationary processes.

\subsection{Couplings of Stationary Processes}

We briefly cover some notation and background on stationary probability measures and processes.
Let $\X$ and $\Y$ be finite sets with their discrete topology, and let $\calU = \X^\bbN$ and $\calV = \Y^\bbN$ be 
associated sequence spaces.  
Elements $\bfx = (x_1, x_2, \ldots) \in \calU$ and $\bfy = (y_1, y_2, \ldots) \in \calV$ are infinite sequence with 
entries in $\X$ and $\Y$, respectively.
For a Borel measure $\mu \in \calM(\X^\bbN)$, let $\mu_k \in \calM(\X^k)$ denote the distribution of the first $k$ coordinates of $\bfx = (x_1, x_2, \ldots)$ under $\mu$, that is, $\mu_k(A) = \mu(A \times \X \times \X \times \cdots)$ for each $A \subseteq \X^k$.
Let $\sigma : \X^\bbN \rightarrow \X^\bbN$ be the left-shift map on $\X^\bbN$ defined by $\sigma(x_1, x_2, ...) = x_2, x_3, \ldots$;
note that $\sigma$ is continuous under the usual product topology on $\X^\bbN$.
A Borel measure $\mu \in \calM(\X^\bbN)$ is said to be {\em stationary} if $\mu \circ \sigma^{-1} = \mu$.
A stationary measure $\mu$ is said to be {\em ergodic} if $\mu(A) \in \{0, 1\}$ for any measurable set $A \subset \X^\bbN$ such that $\sigma^{-1}(A) = A$.
Let $\calM_s(\X^\bbN)$ denote the set of stationary Borel measures on $\X^\bbN$.  
In the same way we may define the left-shift $\tau: \Y^\bbN \rightarrow \Y^\bbN$ and the corresponding set of stationary measures $\calM_s(\Y^\bbN)$ on $\Y^\bbN$.
For the remainder of the paper, we fix the finite spaces $\X$ and $\Y$, 
and consider stationary ergodic measures $\mu \in \calM_s(\X^\bbN)$ and $\nu \in \calM_s(\Y^\bbN)$.

It is helpful to recall the simple equivalence between stationary measures and stationary processes.
The measure $\mu \in \calM_s(\X^\bbN)$ corresponds to a stationary process $X = X_1, X_2, \ldots$ with $X_i \in \X$ via the relation $\mathbb{P}( X_1^k \in A ) = \mu_k(A)$ for all $A \subseteq \X^k$ and all $k \geq 1$.
Ergodicity of $\mu$ is equivalent to ergodicity of the process $X$.  
In the same way, the measure $\nu \in \calM_s(\Y^\bbN)$ corresponds to a stationary process $Y = Y_1, Y_2, \ldots$ with values in $\Y$.  
Each coupling $\pi \in \Pi(\mu,\nu)$ corresponds to a joint process $(\tilde{X},\tilde{Y}) = (\tilde{X}_1,\tilde{Y}_1), (\tilde{X}_2,\tilde{Y}_2), \ldots$ such that $\tilde{X} \eqd X$ and $\tilde{Y} \eqd Y$.
With a slight abuse of notation, we will use $\Pi(X,Y)$ to refer to the set of couplings of stationary processes $X$ and $Y$. 
Importantly, the definition of coupling does {\em not} ensure that the joint process $(\tilde{X},\tilde{Y})$ is stationary, even when
$X$ and $Y$ are (see Examples \ref{ex:nonstationary_coupling} and \ref{ex:ot_is_bad} below).

In what follows we will consider a non-negative, single letter cost function $c: \X \times \Y \rightarrow \bbR_+$ that is defined on pairs of elements from $\X$ and $\Y$.
Note that $c$ is necessarily bounded as $\X$ and $\Y$ are finite. 
By considering sliding blocks, our results may be readily extended to cost functions depending on a finite number of letters.  
Single (and finite) letter cost functions are the norm in information theory, and are natural when making inferences about processes that are only partially observed.
Note also that any cost function $c' :  \X^\bbN \times \Y^\bbN \to \bbR_+$ that is continuous in the usual product topology
can be uniformly approximated by a finite letter cost function.

Any single letter cost function $c: \X \times \Y \rightarrow \bbR_+$ can be extended to a cost function
$c_0 :  \X^\bbN \times \Y^\bbN \to \bbR_+$ on infinite sequences by defining $c_0(\bfx, \bfy) = c(x_1,y_1)$.
In this case, the optimal transport problem for stationary measures $\mu \in \calM_s(\X^\bbN)$ and 
$\nu \in \calM_s(\Y^\bbN)$ can be written as
\begin{eqnarray}
\label{eqn:OTP}
\inf\limits_{\pi \in \Pi(\mu, \nu)} \int c_0 \, d\pi 
\ = \ 
\inf\limits_{\pi \in \Pi(\mu, \nu)} \int c \, d\pi_1
\ = \ 
\inf\limits_{(\tilde{X},\tilde{Y}) \in \Pi(X, Y)} \bbE c(\tilde{X}_1, \tilde{Y}_1)
\end{eqnarray}
where $X$ and $Y$ are the stationary processes associated with $\mu$ and $\nu$ respectively, and $\pi_1$
is the one-dimensional marginal distribution of $\pi$.
In this case we will, with a slight abuse of notation, write $\otc(\mu, \nu)$ instead of $\ot_{c_0}(\mu, \nu)$.

In order to motivate the consideration of stationary couplings, we present two elementary examples.
The first shows that couplings of stationary processes can have very different properties than 
the processes being coupled.

\vskip.1in

\begin{ex}[\cite{o2020optimal}]
\label{ex:nonstationary_coupling}
Let $X$ and $Y$ be Bernoulli$(\nicefrac{1}{2})$ processes, which need not be defined on the 
same probability space.  Let $Z = Z_1, Z_2, \ldots$ be iid Bernoulli$(\nicefrac{1}{2})$ random variables
defined on a common probability space.  Define new processes $\tilde{X} = \tilde{X}_1, \tilde{X}_2, \ldots$ and 
$\tilde{Y} = \tilde{Y}_1, \tilde{Y}_2, \ldots$ in terms of $Z$ as follows: 
\[
\tilde{X}_i = Z_{2i}
\ \ \mbox{ and } \ \ 
\tilde{Y}_i
\ = \
\left\{
	\begin{array}{ll}
		Z_{2i + 1}  & \mbox{if } i \neq 2^k \\[.1in]
		 Z_{2i} & \mbox{if } \exists k, \, i = 2^k.
	\end{array}
\right.
\]
Clearly $\tilde{X} \eqd X$ and $\tilde{Y} \eqd Y$, and therefore the joint process 
$(\tilde{X},\tilde{Y}) = (\tilde{X}_1,\tilde{Y}_1), (\tilde{X}_2,\tilde{Y}_2), \ldots$ is a coupling of $X$ and $Y$.
However, it is easy to see that joint process $(\tilde{X},\tilde{Y})$ is not stationary and that it exhibits long-range dependence.
\end{ex}

\vskip.1in

The next example shows that the optimal transport distance between two stationary processes can be zero even
when the processes are distinct.

\vskip.1in

\begin{ex}
\label{ex:ot_is_bad}
Suppose that $\X = \Y$ and that $c(x,y) = \mathbbm{1}(x \neq y)$ is the 0-1 cost function.  Let $X$ and $Y$ be
two stationary processes that are distinct but share the same one-dimensional distribution.  (For example,
$X$ might be a Bernoulli$(\nicefrac{1}{2})$ process and $Y$ a stationary Markov chain that
cycles between $0$ and $1$.)  
We may couple $X$ and $Y$ as follows.  First, let $\tilde{X}_1 = \tilde{Y_1}$.
For $k \geq 2$ generate $\tilde{X}_{k}$ according to the conditional distribution of $X_k$ given $X_1^{k-1}$
and independently generate $\tilde{Y}_{k}$ according to the conditional distribution of $Y_{k}$ given $Y_1^{k-1}$.
The resulting process $(\tilde{X}, \tilde{Y})$ is a non-stationary coupling of $X$ and $Y$.  Moreover
$\bbE c(\tilde{X}_1, \tilde{Y}_1) = 0$ by design, and therefore the optimal transport cost between $X$ and $Y$ is
zero in this case, despite the fact that $X$ and $Y$ are distinct processes.
\end{ex}

\subsection{Joinings of Stationary Processes}

Example \ref{ex:ot_is_bad} illustrates an important feature of the optimal transport problem for stationary processes: 
for single or finite letter cost functions, an optimal coupling $(\tilde{X}, \tilde{Y})$ of processes $X$ and $Y$ need only 
align the initial components of $\tilde{X}$ and $\tilde{Y}$. 
If, for example, $X$ and $Y$ represent discrete-time audio and video sequences, 
an optimal coupling will seek to align the initial sequence of audio and video, 
but will be insensitive to differences at subsequent time points.
While the use of infinite-letter cost functions can address some of these issues, infinite or long-range costs
can be problematic in practice.
More importantly, as the examples above illustrate, optimal couplings of stationary processes need not be stationary.
From a theoretical and practical perspective, it is natural to consider stationary couplings of stationary processes, 
which are also known as joinings.

\begin{defn}
A probability measure $\lambda$ is a \emph{joining} of $\mu \in \calM_s(\X^\bbN)$ and $\nu \in \calM_s(\Y^\bbN)$ 
if $\lambda$ is a coupling of $\mu$ and $\nu$ and is itself stationary, that is, $\lambda \in \calM_s(\X^\bbN \times \Y^\bbN)$.
The set of joinings of $\mu$ and $\nu$ will be denoted by $\J(\mu, \nu)$.
\end{defn}

Equivalently, a joining of two stationary processes $X$ and $Y$ is a coupling $(\tilde{X}, \tilde{Y})$ that is itself stationary.
This set of processes will be referred to as $\J(X, Y)$.
Note that $\J(\mu, \nu)$ is non-empty, as the independent coupling $\mu \otimes \nu$ is always stationary.
Joinings were introduced by Furstenberg \cite{furstenberg1967disjointness}, and have been studied extensively 
in the ergodic theory literature since that time; 
an overview and more details can be found in \citep{de2005introduction,glasner2003ergodic}.

Restricting the optimal transport problem (\ref{eqn:OTP}) (with a single letter cost) to the set of stationary couplings leads to
the \emph{optimal joining problem}: given measures $\mu \in \calM_s(\X^\bbN)$ and $\nu \in \calM_s(\Y^\bbN)$ find
\begin{equation}
\label{eq:oj_prob}
\ojc(\mu, \nu) 
\ = \
\inf\limits_{\lambda \in \J(\mu, \nu)} \int c \, d\lambda_1
\ = \
\inf\limits_{(\tilde{X},\tilde{Y}) \in \J(X, Y)} \bbE c(\tilde{X}_1, \tilde{Y}_1)
\end{equation}
where $X$ and $Y$ are the stationary processes associated with $\mu$ and $\nu$ respectively.
We will denote the set of joinings attaining the infimum in \eqref{eq:oj_prob} by $\Jmin(\mu, \nu)$.
Elements of $\Jmin(\mu, \nu)$ will be called optimal joinings, and $\ojc(\mu, \nu)$ will be called the optimal joining cost.
Note that when there is no risk of confusion, we will omit the cost $c$ in our notation for the optimal joining cost, writing $\oj(\mu, \nu)$ instead of $\ojc(\mu, \nu)$.
The following proposition collects some standard properties of $\J(\mu, \nu)$ and $\Jmin(\mu, \nu)$; for more details see
\cite{shields1996ergodic} or \cite{mcgoff2021empirical} and the references therein.  

\begin{prop}
\label{prop:properties_of_Jmin}
Under the stated assumptions, the set $\J(\mu, \nu)$ is non-empty, convex, and compact in the weak topology, and
its extreme points coincide with ergodic joinings of $\mu$ and $\nu$.  Moreover, the set
$\Jmin(\mu, \nu)$ of optimal joinings of $\mu$ and $\nu$ is non-empty, convex, and compact in the weak topology, and
its extreme points coincide with the set of ergodic optimal joinings.
\end{prop}

It turns out that there are close connections between the optimal joining problem and the optimal
transport problem using long run average cost.
For $k \geq 1$ let $c_k: \X^k \times \Y^k \rightarrow \bbR_+$ be the $k$-step cumulative cost 
defined by $c_k(x_1^k, y_1^k) = \sum_{\ell = 1}^k c(x_\ell, y_\ell)$, and 
let $\overline{c}(\bfx, \bfy) = \limsup_{k\rightarrow\infty} k^{-1} c_k(x_1^k, y_1^k)$.

\begin{restatable}{prop}{ojisotcbar}
\label{prop:oj_is_otcbar}
If $\X$ and $\Y$ are finite and $\mu \in \calM_s(\X^\bbN)$ and $\nu \in \calM_s(\Y^\bbN)$ are ergodic then
\begin{equation*}
\oj(\mu, \nu) 
=
\lim_{k\rightarrow\infty} \frac{1}{k} \otck(\mu_k, \nu_k)
=
\otcbar(\mu, \nu).
\end{equation*}
\end{restatable}

The first equality in Proposition \ref{prop:oj_is_otcbar} was proven in \cite{gray1975generalization} in the special case 
where $\X = \Y$ and $c$ is a metric; a straightforward extension of their arguments establishes the general case above.
The second equality is proven in Section \ref{sec:proofs} using cyclical monotonicity of optimal couplings.
We remark that the arguments in the proof of Proposition \ref{prop:oj_is_otcbar} do not rely on the finiteness of $\X$ and $\Y$ and may be adapted to the case of Polish spaces and continuous and bounded cost.

Proposition \ref{prop:oj_is_otcbar} shows that the optimal joining cost may be obtained as a limit of $k$-step optimal
transport costs, and that this limit is equal to the optimal transport cost under the long term average cost function.
In this sense, the optimal joining problem seeks couplings that have good long-run average behavior, relative to the single letter cost, 
over the complete history of the joint process.
This is a natural objective when considering optimal transport for stationary processes.
Whereas the limiting average cost $\overline{c}$ may be highly irregular as a function,
the set of joinings $\J(\mu, \nu)$ has a relatively simple structure (e.g., compactness and convexity) 
and leads to an optimization problem that is easier to study.

\begin{rem}
Proposition \ref{prop:oj_is_otcbar} implies that the optimal joining problem $\oj(\mu, \nu)$ satisfies Kantorovich duality with respect to $\overline{c}$.
In particular, 
\begin{equation}\label{eq:dual_oj}
\oj(\mu, \nu) = \sup\limits_{f,g} \left\{ \int f \, d\mu + \int g \, d\nu:  f \oplus g \leq \overline{c}\right\},
\end{equation}
where the supremum is taken over $\mu$-integrable $f: \X^\bbN \rightarrow \bbR$ and $\nu$-integrable $g: \Y^\bbN \rightarrow \bbR$.
The only difference between \eqref{eq:dual_oj} and the standard 
Kantorovich dual problem is the appearance of $\overline{c}$ instead of $c$.
More details on Kantorovich duality can be found in \cite{villani2008optimal} and \cite{santambrogio2015optimal}.
While we do not make use of the dual optimal joining problem in this paper, we expect that it may be of use in future analyses.
\end{rem}

\paragraph{Related Work.}
A special case of the optimal joining cost first appeared in the work of Ornstein under the name $\overline{d}$-distance \citep{ornstein1973application}.
For finite alphabets, the $\overline{d}$-distance is equivalent to the optimal joining cost with respect to the discrete metric.
Subsequently, \cite{gray1975generalization} generalized the $\overline{d}$-distance to Polish alphabets and metric costs.
More recently, \cite{ruschendorf2012optimal} considered the same problem for a specific class of examples in which the optimal joining could be determined exactly and extended the problem to random fields.
Subsequent work has considered various aspects of stationary optimal transport in the context of dynamical systems such as duality \citep{zaev2015monge}, decompositions \citep{zaev2016ergodic}, and the existence of optimal transport maps \citep{kolesnikov2017optimal}.
Another line of work \citep{jenkinson2006ergodic,jenkinson2019ergodic}, referred to as ergodic optimization, studies optimization of linear functionals over the set of invariant measures.
One may regard the optimal joining problem as a constrained ergodic optimization problem.
Recent work in statistical inference \citep{mcgoff2020empirical,mcgoff2021empirical,mcgoff2019gibbs} has shown that stationary optimal transport problems arise naturally in the context of fitting dynamical models.
In another direction, \cite{o2020optimal} studied computational aspects of a constrained form of the optimal joining problem for Markov chains.
This constrained optimal joining problem was applied to the comparison and alignment of graphs in \cite{oconnor2021graph}.

\section{Estimation of an Optimal Joining and its Expected Cost}\label{sec:estimate_oj}

In the rest of the paper, we consider the optimal joining problem in a statistical setting, focusing on the problem of estimating an optimal joining and the optimal joining cost of two stationary, ergodic processes from finite sequences of observations.
Fix finite sets $\X$ and $\Y$, a single-letter cost $c: \X \times \Y \rightarrow \bbR_+$, and stationary, ergodic process measures $\mu \in \calM_s(\X^\bbN)$ and $\nu \in \calM_s(\Y^\bbN)$ as in the previous section.
Let $X = X_1, X_2, ...$ and $Y = Y_1, Y_2, ...$ be two processes, possibly defined on different probability spaces, with distributions $\mu$ and $\nu$, respectively, and suppose that we observe $X_1, ..., X_n$ and $Y_1, ..., Y_n$.
The task of interest is to estimate an optimal joining $\lambda \in \Jmin(\mu, \nu)$ and the optimal joining cost $\oj(\mu, \nu)$ from the observed sequences.
In this section, we define the proposed estimates and state our main result regarding their consistency in the limit $n \rightarrow \infty$.

Before describing our proposed estimation scheme in detail, we first provide some intuition.
Proposition \ref{prop:oj_is_otcbar} ensures that the optimal transport cost between the $k$-dimensional 
distributions of $\mu$ and $\nu$ converges to the optimal joining cost as $k$ tends to infinity.
Thus when $k$ is large we expect that a good estimate of $k^{-1} \otck(\mu_k, \nu_k)$ will approximate the optimal joining cost $\oj(\mu, \nu)$.
Furthermore, a stationary coupling achieving the $k$-step optimal joining cost should be 
close to the set of optimal joinings $\Jmin(\mu, \nu)$.  In the setting of interest to us, we do not have access to
the finite dimensional distributions of the observed processes; instead we estimate these from the available observations.
In what follows, let $k$ be a fixed integer between $1$ and $n$.  
The choice of $k$ is discussed below in Section \ref{sec:choiceofk}.

We propose an estimation scheme that is comprised of three steps.
First, we construct $k$-block empirical measures from the available observations of each process.
The resulting probability measures act as empirical estimates of $\mu_k$ and $\nu_k$.
Second, we select an optimal transport plan between these empirical $k$-block measures with respect to $c_k$.
The expected cost of this coupling acts as an estimate of $\otck(\mu_k, \nu_k)$.
Finally, we construct a stationary process measure from the coupling in the second step.
This is done via a $k$-block process construction, described formally below in Definition \ref{defn:block_process}.

\begin{enumerate}[start=1, label={\bfseries Step \arabic*:}, leftmargin=2cm]
\item (Empirical $k$-block measure)
Define the probability measure $\hat{\mu}_{k,n} := \hat{\mu}_k[X_1^n] \in \calM(\X^k)$ by
\begin{equation*}
\hat{\mu}_{k,n}(x_1^k) = \frac{1}{n-k+1} \sum\limits_{\ell=0}^{n-k} \1(x_1^k = X_{\ell+1}^{\ell+k}),
\end{equation*}
\noindent 
for $x_1^k \in \mathcal{X}^k$.
This is referred to as the \emph{$k$-block empirical measure} constructed from observations $X_1^n$.
Let $\hat{\nu}_{k,n}$ be the $k$-block empirical measure constructed from $Y_1^n$ in the same manner.
Note that $\hat{\mu}_{k,n}$ is a probability measure on $\X^k$, and is an estimate of $\mu_k$, the $k$-dimensional
distribution of $\mu$;
analogous remarks apply to $\hat{\nu}_{k,n}$.

\item (Optimal coupling of $k$-block measures)
Find an optimal coupling of $\hat{\mu}_{k,n}$ and $\hat{\nu}_{k,n}$ with respect to $c_k$.
Formally, let $\pi_k \in \Pi(\hat{\mu}_{k,n}, \hat{\nu}_{k,n})$ be any coupling such that
$\int c_k \, d\pi_k \leq \int c_k \, d\pi'_k$ for all $\pi'_k \in \Pi(\hat{\mu}_{k,n}, \hat{\nu}_{k,n})$.
Thus $\int c_k \, d\pi_k = \otck(\hat{\mu}_{k,n}, \hat{\nu}_{k,n})$.
To simplify notation in what follows, define
\begin{equation}
\label{eq:est_oj_cost}
\ojhatkn \ = \ \ojhatk(X_1^n, Y_1^n) \ = \ k^{-1}\otck(\hat{\mu}_{k,n}, \hat{\nu}_{k,n}).
\end{equation}
\end{enumerate}

We regard the number $\ojhatkn$ defined in \eqref{eq:est_oj_cost} as an estimate of the optimal joining cost of the observed processes $X$ and $Y$.
We estimate an optimal joining of these processes by constructing a stationary process measure from the optimal coupling $\pi_k$ appearing in Step 2 above.
In order to do this, we leverage a well-known construction described in Definition \ref{defn:block_process} below.

\begin{defn}[Block process construction]
\label{defn:block_process}
Let $\calU$ be finite and $k \geq 1$.
Define $\tilde{\Lambda}^k: \calM(\calU^k) \rightarrow \calM(\calU^\bbN)$ 
to be the map that takes a probability measure $\gamma \in \calM(\calU^k)$ to the unique probability measure on $\calU^\bbN$ obtained by independently concatenating $\gamma$ with itself infinitely many times.
Formally, for any $\ell k$-dimensional cylinder set $C = C_1 \times \cdots \times C_{\ell k} \times \calU \times \cdots \subset \calU^\bbN$,
\begin{align*}
\tilde{\Lambda}^k[\gamma](C) &= \gamma(C_1 \timesdots C_k) \gamma(C_{k+1} \timesdots C_{2k}) \cdots \gamma(C_{(\ell-1)k+1} \timesdots C_{\ell k}) \\
&= \prod\limits_{i=0}^{\ell-1} \gamma(C_{ik+1}^{ik+k}).
\end{align*}
\noindent Moreover, define $\Lambda^k : \calM(\calU^k) \rightarrow \calM_s(\calU^\bbN)$ to be the map defined by randomizing the start of the output of $\tilde{\Lambda}^k$ over the first $k$ coordinates.
Formally, for any set $A \subset \calU^\bbN$,
\begin{equation*}
\Lambda^k[\gamma](A) = \frac{1}{k} \sum\limits_{\ell=0}^{k-1} \tilde{\Lambda}^k[\gamma](\calU^\ell \times A).
\end{equation*}
\noindent We will refer to $\tilde{\Lambda}^k[\gamma]$ as the \emph{independent $k$-block process induced by $\gamma$} 
and $\Lambda^k[\gamma]$ as the \emph{stationary $k$-block process induced by $\gamma$}.
\end{defn}


The block process construction described above is standard in ergodic theory \cite{ornstein1990sampling, shields1996ergodic}, and ensures that the process $\Lambda^k[\gamma]$ is stationary.

\begin{enumerate}[start=3, label={\bfseries Step \arabic*:}, leftmargin=2cm]
\item (Construct stationary process)
Given the $k$-dimensional measure $\pi_k$ obtained in Step 2, define the stationary joint measure 
$
\estojkn \ = \ \Lambda^k[\pi_k] \ \in \ \calM_s(\X^\bbN \times \Y^\bbN).
$
\end{enumerate}

We regard $\estojkn$ as an estimate of an optimal joining of the observed processes.
We establish in Appendix \ref{app:estimate_properties} that $\estojkn$ is a joining of empirical estimates $\Lambda^k[\hat{\mu}_{k,n}]$ and $\Lambda^k[\hat{\nu}_{k,n}]$ of $\mu$ and $\nu$ with expected cost equal to $\ojhatkn$.

\subsection{Consistency}
Having detailed the proposed estimators $\estojkn$ and $\ojhatkn$, we now consider their behavior as the length $n$ of the observed sequences goes to infinity.
Intuitively, for fixed $k$ we expect that when $n$ is large, $\ojhatkn$ will be close to $k^{-1} \otck(\mu_k, \nu_k)$.
However, as Proposition \ref{prop:oj_is_otcbar} suggests, this quantity will only be close to the optimal joining cost when $k$ is large.
Thus in order for our estimates to converge to the desired targets, we must let $k$ grow with $n$.
In particular, we consider sequences of estimates $\{\estojknn\}_{n\geq 1}$ and $\{\ojhatknn\}_{n \geq 1}$ 
for some sequence $k(n)$ such that $k(n) \rightarrow \infty$ and ask whether the two sequences converge to an optimal joining and the optimal joining cost, respectively.
We show in Theorem \ref{thm:consistent_estimation_of_oj} that under the stated assumptions, such a sequence $k(n)$ necessarily exists.
We will say that a sequence of Borel probability measures $\gamma^1, \gamma^2, \ldots \in \calM(\calU)$ 
converges weakly to a set $\Gamma \subset \calM(\calU)$, written as $\gamma^n \Rightarrow \Gamma$, if every subsequence of $\gamma^n$ contains a further subsequence that converges weakly to an element of $\Gamma$.
Moreover, to simplify notation going forward, we will occasionally write $k$ for $k(n)$ when there is no risk of confusion.

\begin{restatable}[]{thm}{consistentestimationofoj}
\label{thm:consistent_estimation_of_oj}
Let $\X$ and $\Y$ be finite and $\mu \in \calM_s(\X^\bbN)$ and $\nu \in \calM_s(\Y^\bbN)$ be ergodic.
Then there exists a sequence $k = k(n)$ with $k(n)\rightarrow \infty$ such that 
$\ojhatkn \rightarrow \oj(\mu, \nu)$ and $\estojkn \Rightarrow \Jmin(\mu, \nu)$ almost surely as $n \rightarrow \infty$.
\end{restatable}

The problem of estimating optimal joinings appears to have not been considered explicitly in the literature.
However, the special case of estimating the $\overline{d}$-distance between two ergodic processes from finite observations has been considered.
The focus of this line of work has been in finding universal estimation schemes, including choices of the sequence $k(n)$, such that the desired convergence holds uniformly over all pairs $\mu$ and $\nu$ from some set.
The estimate $\ojhatkn$ for the optimal joining cost that we propose is an extension of that proposed in \cite{ornstein1990sampling} for the $\overline{d}$-distance.
In the case that $\X = \Y$ and $c$ is the 0-1 cost, it was shown that $\ojhatkn$ with $k = \calO(\log n)$ converges to the optimal joining cost of $\mu$ and $\nu$ whenever $\mu$ is a stationary coding of an iid process.
Later work studied the limits of this estimation scheme \cite{marton1994entropy} and the properties of processes for which the scheme is consistent \cite{ornstein1994d}.
In this context, our consistency results allow for relatively weak assumptions on $\mu$ and $\nu$ (ergodicity) at the expense of loss of control over the sequence $k(n)$.

\begin{rem}
The arguments underlying the proof of Theorem \ref{thm:consistent_estimation_of_oj} may be adapted in a straightforward way to show that the proposed estimates are consistent more generally whenever $\X$ and $\Y$ are compact and $c$ is continuous.
\end{rem}

\subsection{Choice of $k(n)$} 
\label{sec:choiceofk}
Theorem \ref{thm:consistent_estimation_of_oj} raises the question of how the sequence $k(n)$ depends on 
the processes $X$ and $Y$.
In the proof, we find that the choice of the sequence $k(n)$ is related to a notion of admissibility, which we 
now describe.

\begin{defn}\label{defn:adapted_cost}
Let $\X$ and $\Y$ be finite and let $c: \X \times \Y \rightarrow \bbR_+$ be a cost function.
Then the $\X$-adapted cost $\cX: \X \times \X \rightarrow \bbR_+$ is defined by
\begin{equation*}
\cX(x, x') = \sup\limits_{y \in \Y} |c(x, y) - c(x', y)|,
\end{equation*}
with the $\Y$-adapted cost $\cY: \Y \times \Y \rightarrow \bbR_+$ defined in the analogous way.
\end{defn}

The adapted cost arises naturally when studying the Lipschitz properties of the optimal transport and optimal joining costs (see Lemmas \ref{lemma:eot_is_lipschitz} and \ref{lemma:oj_is_lipschitz}).
Define $\cXk: \X^k \times \X^k \rightarrow \bbR_+$ to be the sum of the 
$\X$-adapted cost over $k$ coordinates, and define $\cYk$ similarly.
Note that $(\X^k, \cXk)$ and $(\Y^k, \cYk)$ 
are well-defined pseudometric spaces for every $k \geq 1$.

\begin{defn}\label{def:c_admissibility}
Let $\X$ and $\Y$ be finite and let $c: \X \times \Y \rightarrow \bbR_+$ be a cost function.
We will say that a nondecreasing sequence $k = k(n)$ with $k \rightarrow \infty$ is $c$-\emph{admissible} for $\mu \in \calM_s(\X^\bbN)$ if
\begin{equation*}
\mu\left(\bfx \in \X^\bbN: \lim\limits_{n\rightarrow\infty}\frac{1}{k}\otcXk\left(\hat{\mu}_{k,n}, \mu_{k}\right) = 0 \right) = 1.
\end{equation*}
\noindent We define $c$-admissibility for $\nu \in \calM_s(\Y^\bbN)$ in the analogous way.
\end{defn}

The $c$-admissibility of a sequence $k = k(n)$ ensures that the average transport distance between the $k$-dimensional distributions of $\mu$ and their empirical counterparts tends to zero almost surely as the number of observations $n$ increases.
The property $c$-admissibility weakens the notions of admissibility (with respect to total variation distance) and $\overline{d}$-admissibility discussed in \cite{shields1996ergodic}.
We note that certain sequences growing like $\calO(\log n)$ are known to be admissible (and thus $c$-admissible) for aperiodic Markov chains \citep{marton1994entropy}.

\begin{prop}\label{prop:admissibility_consistency}
Under the hypotheses of Theorem \ref{thm:consistent_estimation_of_oj}, if a sequence $k = k(n)$ is $c$-admissible for both $\mu$ and $\nu$, then with $\mathbb{P}$-probability one, $\ojhatkn \rightarrow \oj(\mu, \nu)$ and $\estojkn \Rightarrow \Jmin(\mu, \nu)$ as $n \rightarrow \infty$.
\end{prop}

\begin{ex}
Let $\mu$ and $\nu$ be aperiodic Markov chains with entropy rates $h(\mu)$ and $h(\nu)$, respectively (defined in Section \ref{sec:entropic_oj}).
Then for any $\varepsilon > 0$, the sequence $k(n) =  (h(\mu) \vee h(\nu) + \varepsilon)^{-1} \log n$ is admissible (and thus $c$-admissible) for both $\mu$ and $\nu$ \citep{marton1994entropy}.
Thus by Proposition \ref{prop:admissibility_consistency}, $\ojhatkn \rightarrow \oj(\mu, \nu)$ and $\estojkn \Rightarrow \Jmin(\mu, \nu)$ as $n \rightarrow \infty$.
\end{ex}

\begin{rem}
Other stationary process constructions may be used in Step 3 of the proposed estimation scheme.
For example, one may construct a finite-order, stationary Markov process from $\pi_k$.
However, the expected cost of a stationary Markov chain constructed from $\pi_k$ will generally not be equal to $\ojhatkn$ and may require more care to control.
On the other hand, the approach detailed in Step 3 enables us to control the expected cost of the constructed process $\estojkn$, namely $\ojhatkn$, which we show converges to the optimal joining cost.
\end{rem}

\begin{rem}
The optimal transport distance between $k$-block measures was proposed as an extension of optimal transport to stationary time series in \cite{muskulus2011wasserstein}.
However, that work did not consider the relationship of this approach to the optimal joining problem or the consistency of the proposed distance.
\end{rem}

\section{Finite-Sample Error Bound}\label{sec:rates}

In this section, we provide an upper bound on the expected error of the proposed estimate of the optimal joining cost when the observed
processes satisfy suitable mixing conditions.
Our bounds are derived from bounds on the optimal transport cost between a measure and an estimate of that measure based on a finite number of samples detailed in \cite{boissard2014mean}.
A substantial body of work has considered this problem for iid processes from both asymptotic and finite-sample perspectives \citep{dudley1969speed, boissard2014mean, fournier2015rate, weed2019sharp, mena2019statistical, genevay2018sample, klatt2020empirical}.
Other work has focused on rates of convergence and central limit theorems for the 1-Wasserstein distance ($\ot_d$ for a metric $d$) when samples are drawn from a stationary process satisfying a certain mixing condition \citep{dede2009empirical, boissard2014mean, dedecker2017behavior, berthet2020central}.
In order to obtain an error bound for the estimated optimal joining cost, we consider the case that the marginal processes $\mu$ and $\nu$ are $\phi$-mixing.


\begin{defn}
Let $\calU$ be finite.
A process measure $\gamma \in \calM_s(\calU^\bbN)$ has $\phi$-mixing coefficients 
$\phi_\gamma: \bbN_0 \rightarrow \bbR_+$ where $\phi_\gamma(0) = 1$, and for any $g \geq 0$,
\begin{equation*}
\phi_\gamma(g+1) = \sup\left\{|\gamma(\calU^{g} \times B | A) - \gamma(B)|: A \subset \calU^\ell,  \ell \geq 1, B \subset \calU^\bbN\right\},
\end{equation*}
\noindent where $\gamma(\calU^{g} \times B | A) = \gamma(A \times \calU^{g} \times B)/\gamma(A)$.
The measure $\gamma$ is called $\phi$-\emph{mixing} if $\lim_{g\rightarrow\infty} \phi_\gamma(g) = 0$.
\end{defn}

\noindent The $\phi$-mixing condition is a standard strong mixing condition in the study of stochastic processes.
Existing work on estimation of optimal transport costs under dependence has generally focused on weaker mixing conditions such as $\alpha$-mixing.
We find $\phi$-mixing to be particularly suited to our arguments in proving an error bound for the estimated optimal joining cost.
For more details on $\phi$-mixing and its relationship to other strong mixing conditions, we refer the reader to \cite{bradley2005basic}.

For a pseudometric space $(\calU, d)$, let $\mathcal{N}(\calU, d, \varepsilon)$ denote the $\varepsilon$-covering number of $\calU$ with respect to the pseudometric $d$.
We now present our finite sample error bound.

\begin{restatable}[]{thm}{abstractrates}
\label{thm:abstract_rates}
Let $\mu$ and $\nu$ have $\phi$-mixing coefficients $\phi_\mu$ and $\phi_\nu$, respectively.
Then there exists a constant $C < \infty$ such that for every $n \geq 1$, $k \in \{1, ..., n\}$, $g \geq 0$ and $t \in (0, \frac{1}{4} \|c\|_\infty]$,
\begin{align*}
\mathbb{E}\left|\ojhatkn - \oj(\mu, \nu)\right| 
\, \leq \, 
\|c\|_\infty \left(\frac{k (\phi_\mu(g+1) + \phi_\nu(g+1))}{k + g} + \frac{3 g}{k}\right) + C \left(t + u_t(k, n) + v_t(k, n)\right), \\[.08in]
\end{align*}
where
\begin{equation*}
u_t(k, n) \, = \,
\left(\frac{1}{n^2}\sum\limits_{\ell=0}^n (n-\ell+1) \phi_\mu^{\nicefrac{1}{2}}(\ell) \right)^{\nicefrac{1}{2}} \int_t^{\frac{1}{4}\|c\|_\infty} \mathcal{N}\left(\X^k, \frac{1}{k} \cXk, \varepsilon\right)^{\nicefrac{1}{2}} \, d\varepsilon, \\[.08in]
\end{equation*}
and $v_t(k, n)$ is defined similarly in terms of $\phi_\nu$ and $\Y$.
\end{restatable}

Theorem \ref{thm:abstract_rates} gives a general upper bound on the expected error of the estimate of the optimal joining 
cost in terms of the $\phi$-mixing coefficients of $\mu$ and $\nu$ and the covering numbers of the product spaces 
$\X^k$ and $\Y^k$ with respect to $\cXk$ and $\cYk$.
When the cost $c$ is less variable the covering numbers under $\cX$ and $\cY$ will 
be smaller, and the upper bound of the theorem will be smaller as well.

\begin{restatable}[]{cor}{rates}
\label{cor:rates}
Let $\mu$ and $\nu$ have $\phi$-mixing coefficients $\phi_\mu$ and $\phi_\nu$, respectively, satisfying 
\begin{equation*}
\sum_{\ell=0}^n (n-\ell) \phi_\mu^{\nicefrac{1}{2}}(\ell) = \calO(n^p) \quad\quad\mbox{and}\quad\quad \sum_{\ell=0}^n (n- \ell) \phi_\nu^{\nicefrac{1}{2}}(\ell) = \calO(n^p)
\end{equation*}

\noindent for some $p \in [1, 2)$.
Then there exists a constant $C < \infty$ depending only on $\phi_\mu$ and $\phi_\nu$ such that for every $k \geq 1$, $g \geq 0$, and $n$ large enough,
\begin{equation}
\label{eq:error_bound}
\mathbb{E}\left|\ojhatkn - \oj(\mu, \nu)\right| \leq \|c\|_\infty\left(\frac{k (\phi_\mu(g+1) + \phi_\nu(g+1))}{k + g} + \frac{3 g}{k} + \frac{C (|\X|^{\nicefrac{k}{2}} + |\Y|^{\nicefrac{k}{2}})}{n^{1 - \nicefrac{p}{2}}}\right).
\end{equation}

\noindent In particular, if $k(n) < \frac{(2-p) \log n}{\log(|\X| \vee |\Y|) \vee 1}$ and $g(n) = o(k(n))$ with $k(n), g(n) \rightarrow \infty$, then the upper bound converges to zero as $n \rightarrow\infty$.
\end{restatable}

Corollary \ref{cor:rates} provides finite-sample control of the mean error in the estimated optimal joining cost.
In particular, it sheds some light on how the choice of block size $k$ interacts with the amount of dependence of the marginal processes (as quantified by their $\phi$-mixing coefficients) and the sample size $n$.
Previous work \citep{csiszar2010rate, talata2013divergence, talata2010divergence, gallo2011markov, bressaud1999speed} has established error bounds and rates of convergence for Markov approximations to ergodic processes.
However, it appears that no previous work has established such results for the $k$-block process estimate.
We remark that Theorem \ref{thm:abstract_rates} and Corollary \ref{cor:rates} include the special case of Ornstein's 
$\overline{d}$-distance and thus provide some additional insight into the estimation scheme for this distance proposed in \cite{ornstein1990sampling}.
To provide further context for Corollary \ref{cor:rates}, we consider two examples below.

\begin{ex}[IID Processes]
If $\mu$ and $\nu$ are iid processes, then $\oj(\mu, \nu) = \otc(\mu_1, \nu_1)$.
Moreover, $\phi_\mu(g) = \phi_\nu(g) = 0$ for every $g \geq 1$, and so we may let $k = 1$, $g = 0$, and $p = 1$.
Then by Corollary \ref{cor:rates}, we see that
\begin{equation*}
\mathbb{E}\left| \ojhat_{1,n} - \oj(\mu, \nu)\right| = \mathcal{O}\left(n^{-\nicefrac{1}{2}}\right). \\[.08in]
\end{equation*}
The rate above is consistent with known rates for the estimation of the 1-Wasserstein distance on finite spaces \citep{boissard2014mean}.
\end{ex}

For iid processes, the optimal joining problem reduces to the optimal coupling problem of their $1$-dimensional 
marginal measures. However, when at least one of the measures is not iid, the optimal joining need not be 
Markov of any order \citep{ellis1976thedj}, and one must let $k$ tend to infinity in order to estimate the full 
behavior of an optimal joining.  As such, one expects to find slower rates outside the iid setting.

\begin{ex}[Markov Processes]
\label{ex:markov_bound}
If $\mu$ and $\nu$ are aperiodic irreducible Markov chains, then there exist constants $C < \infty$ and $\rho \in (0,1)$ such that $\phi_\mu(g) \leq C \rho^g$ and $\phi_\nu(g) \leq C \rho^g$ \citep{davydov1974mixing, bradley2005basic} for each $g \geq 1$.
Thus the summability conditions in Corollary \ref{cor:rates} are satisfied with $p = 1$.
Applying Corollary \ref{cor:rates} with $k(n) = \left\lfloor\frac{\alpha \log(n)}{\log(|\X| \vee |\Y|) \vee 1}\right\rfloor$ 
and $g(n) = \left\lfloor\frac{\log(\alpha \log(n))}{\log(\nicefrac{1}{\rho})}\right\rfloor$ for any $\alpha \in (0,1)$ and 
$n$ large enough, we find
\begin{equation}
\label{eq:markov_bound}
\mathbb{E}\left|\ojhatkn - \oj(\mu, \nu)\right| = \mathcal{O}\left(\frac{\log(\log(n))}{\log(n)}\right).
\end{equation}
\end{ex}

It was established in \cite{marton1994entropy} that any sequence $k(n)$ growing faster than $\log n$ is inadmissible for every ergodic process with positive entropy rate.
Thus the sequence $k(n)$ in Example \ref{ex:markov_bound} is the best achievable rate (up to constant factors) for estimating the finite-dimensional distributions of aperiodic irreducible Markov chains in general.
However, the rate \eqref{eq:markov_bound} is substantially slower than the polynomial rate typically observed when estimating marginal distributions of Markov chains in, for example, total variation distance \citep{boissard2014mean}.
This disparity is due to the $\nicefrac{g}{k}$ term in \eqref{eq:error_bound}, which arises when joining approximations of the processes $\mu$ and $\nu$ in the proof of Theorem \ref{thm:abstract_rates}.

\section{The Entropic Optimal Joining Problem}\label{sec:entropic_oj}

A large body of recent work in optimal transport has focused on studying the computational and statistical properties of regularized versions of the optimal transport problem.
Entropic regularization in particular has attracted a great deal of interest from the machine learning and statistics communities as a means of smoothing the optimal transport problem and enabling more efficient computation of solutions.
For any $\eta > 0$, the \emph{entropic optimal transport problem} is obtained by subtracting the Shannon entropy $H(\pi) = -\sum_{u,v} \pi(u,v) \log \pi(u,v)$ from the optimal transport objective:
\begin{equation*}
\eotc(\mu, \nu) = \inf\limits_{\pi \in \Pi(\mu, \nu)} \left\{\int c \, d\pi - \eta H(\pi)\right\}.
\end{equation*}

Cuturi \citep{cuturi2013sinkhorn} showed that solutions to this problem have a special form and can be obtained via a matrix scaling method known as the Sinkhorn-Knopp algorithm \citep{sinkhorn1967diagonal}.
Subsequent work \citep{altschuler2017near, dvurechensky2018computational, lin2019efficient, lin2019acceleration, guo2020fast} has analyzed the complexity of this and other algorithms for entropic optimal transport in detail, showing that approximations of the optimal transport cost may be obtained in time that is nearly-linear in the dimension of the couplings under consideration.
Other work \citep{klatt2020empirical, genevay2018sample, mena2019statistical, hundrieser2021entropic} has studied the entropic problem from a statistical perspective, proving estimation error bounds and central limit theorems for the empirical entropic optimal transport cost.
In some cases, these estimates exhibit better sample complexity and faster rates of convergence than the best achievable quantities for the unregularized problem, which are known to suffer from the ``curse of dimensionality".
Recent work has found closed form solutions to the entropic problem between Gaussian distributions \citep{del2020statistical, tong2021entropy, janati2020entropic} and examined the convergence of solutions to solutions of the unregularized problem as the regularization coefficient $\eta$ converges to zero \citep{peyre2019computational, bernton2021entropic, nutz2021entropic}.
For additional details on entropic optimal transport, we refer the reader to \cite{peyre2019computational}.

In this section, we consider entropic regularization of the optimal joining problem.
We identify a natural penalty term for the optimal joining problem by viewing the regularized problem as a 
limit of entropic optimal transport problems with increasing dimension.  We observe that these problems converge 
to a regularized optimal joining cost with the \emph{entropy rate} as the penalty term.
Entropy rates have been studied in the context of stochastic processes and 
information theory for many years, dating back to Shannon \citep{shannon1948mathematical}.

\begin{defn}
Let $\calU$ be finite and let $\gamma \in \calM_s(\calU^\bbN)$ be a stationary measure.
For $k \geq 1$ define $H(\gamma_k) = -\sum_{u_1^k} \gamma_k(u_1^k) \log \gamma_k(u_1^k)$.
The \emph{entropy rate} of $\gamma$ is defined by $h(\gamma) := \lim_{k\rightarrow\infty} \frac{1}{k}H(\gamma_k)$.
\end{defn}

In other words, the entropy rate is the limiting joint entropy per symbol of the finite dimensional distributions of the process.
By subadditivity, the limit in the definition exists and is equal to $\inf_{k \geq 1} \frac{1}{k} H(\gamma_k)$. 
An iid process with one dimensional distribution $p$ has entropy rate equal to $H(p)$.
A stationary, aperiodic, irreducible Markov chain with stationary distribution $p$ and transition matrix $P$ 
has entropy rate given by $-\sum_{ij} p_i P_{ij} \log P_{ij}$.
Occasionally, we will use $H_k(\cdot)$ instead of $H(\cdot)$ to emphasize the dependence on the dimension $k$.
We remark that the entropy rate is known to be weakly upper semicontinuous on finite-alphabet sequence spaces \citep{walters2000introduction}.
We make use of this fact in Section \ref{sec:proofs_eoj}, for example, when establishing the consistency of our entropically regularized estimates, defined below.

\begin{rem}
The entropy rate of certain processes is preserved under randomization, as in Definition \ref{defn:block_process}.
In particular, if a process $\tilde{\gamma}$ is $k$-stationary for any $k \geq 1$ and $\gamma$ is the stationary process obtained by randomizing the start of $\tilde{\gamma}$, then $h(\tilde{\gamma}) = h(\gamma)$.
A formal statement of this fact along with a proof may be found in Appendix \ref{app:entropy}.
As our proposed estimate of an optimal joining is constructed by randomizing the start of a $k$-stationary process, 
this fact is needed when considering the entropy rate of the estimate.
\end{rem}

For $\eta > 0$, we define the \emph{entropic optimal joining problem} by
\begin{equation}\label{eq:entropic_oj}
\eojc(\mu, \nu)
\ = \
\inf\limits_{\lambda \in \J(\mu, \nu)} \left\{ \int c \, d\lambda_1 - \eta h(\lambda)\right\}.
\end{equation}

\noindent As in the unregularized problem, one can show under the stated assumptions that the infimum in \eqref{eq:entropic_oj} is attained.
We include a proof of this fact in Appendix \ref{app:existence}.
From now on, we will denote the set of joinings achieving the infimum in \eqref{eq:entropic_oj} by $\Jmin^\eta(\mu, \nu)$.
Moreover, we will drop the cost $c$ in our notation, writing $\eoj(\cdot, \cdot)$ for $\eojc(\cdot, \cdot)$ when there is no risk of confusion.
Note that \eqref{eq:entropic_oj} is still well-defined when $\eta = 0$ but in that case, we recover the standard optimal joining problem and thus refer to it by that name.

Proposition \ref{prop:oj_is_otcbar} shows that the optimal joining cost may be obtained as a limit of $k$-step optimal transport costs.
The next proposition extends this result to the entropic optimal joining problem.

\begin{restatable}[]{prop}{convergenceofeottoeoj}\label{prop:convergence_of_eot_to_eoj}
Let $\X$ and $\Y$ be finite and $\mu \in \calM_s(\X^\bbN)$ and $\nu \in \calM_s(\Y^\bbN)$.
Then for any $\eta \geq 0$, 
\begin{equation*}
\lim\limits_{k\rightarrow\infty} \frac{1}{k} \eotck(\mu_k, \nu_k) 
\ = \
\eoj(\mu, \nu).
\end{equation*}
\end{restatable}

\noindent 
In other words, the entropic optimal joining cost can be obtained as a limit of the average entropic optimal transport costs. 
Thus the entropy rate emerges as a natural regularizer for the optimal joining problem.
Moreover, the proposition suggests that the approach to estimating an optimal joinings may extend to the regularized setting.
In particular, when $k$ is large, a good estimate of $k^{-1} \eotck(\mu_k, \nu_k)$ should be a good estimate of $\eoj(\mu, \nu)$,
and existing algorithms for efficient computation of $k^{-1} \eotck(\mu_k, \nu_k)$ will translate to faster estimates of $\eoj(\mu, \nu)$.

Entropic penalization has appeared in some related work.
In \cite{mcgoff2019gibbs}, a regularized optimal joining problem with fiber entropy as the penalty term arises naturally in the context of Bayesian estimation of dynamical models.
Another line of work \citep{lopes2012entropy,lopes2015entropy} has investigated the role of entropic regularization in the thermodynamic formalism, which consists of an optimization of a linear functional over the set of invariant measures.
More recently, \cite{o2020optimal} proposed an extension of entropic techniques to a constrained optimal transport problem specifically for Markov chains.
In that work, the entropy constraint implies a constraint on the entropy rate of the set of joinings in question and leads to improved computational efficiency.
Existing work has yet to propose a principled regularization scheme for the optimal joining problem considered here.

Before moving on to the proposed estimation scheme, we consider the stability of the entropic optimal joining cost in $\eta$, and
in particular, the limiting behavior of $\eoj(\mu, \nu)$ when $\eta$ tends to zero.

\begin{restatable}[]{prop}{convergenceineta}
\label{prop:convergence_in_eta}
Let $\X$ and $\Y$ be finite and $\mu \in \calM_s(\X^\bbN)$ and $\nu \in \calM_s(\Y^\bbN)$.
Then the entropic optimal joining cost satisfies
\begin{equation*}
\lim_{\eta \rightarrow 0} \eoj(\mu, \nu) = \oj(\mu, \nu).
\end{equation*}
\end{restatable}

Thus the entropic optimal joining cost converges to the unregularized optimal joining cost as the penalty parameter shrinks.
This behavior is consistent with analogous results in optimal transport \citep{peyre2019computational}.

\subsection{Extension of the Estimation Procedure}
The proposed estimation scheme (described in Section \ref{sec:estimate_oj}) may be easily extended to the entropic optimal joining problem.
One need only consider a modification of Step 2 in which one solves an entropic optimal transport problem for $\eta > 0$:

\begin{enumerate}[start=2, label={\bfseries Step \arabic*':}, leftmargin=2cm]
\item (Entropic optimal coupling) 
Find an entropic optimal coupling of $\hat{\mu}_{k,n}$ and $\hat{\nu}_{k,n}$.
Formally, let $\pi_k \in \Pi(\hat{\mu}_{k,n}, \hat{\nu}_{k,n})$ be any coupling such that $\int c_k \, d\pi_k - \eta H(\pi_k) \leq \int c_k \, d\pi'_k - \eta H(\pi'_k)$ for all $\pi'_k \in \Pi(\hat{\mu}_{k,n}, \hat{\nu}_{k,n})$.
%
Thus, $\pi_k$ has expected entropic $k$-step cost equal to $\eotck(\hat{\mu}_{k,n}, \hat{\nu}_{k,n})$.
To simplify notation in what follows, define
\begin{equation}
\label{eq:est_eoj_cost}
\eojhatetakn \ = \ \eojhatetak(X_1^n, Y_1^n) \ = \ k^{-1} \eotck(\hat{\mu}_{k,n}, \hat{\nu}_{k,n}).
\end{equation}
\end{enumerate}

One may then construct a stationary process from $\pi_k$ via the block process construction detailed in Definition \ref{defn:block_process}.
In particular, define $\esteojetakn = \Lambda^k[\pi_k] \in \calM_s(\X^\bbN \times \Y^\bbN)$.
We propose $\esteojetakn$ as an estimate of an entropic optimal joining of $\mu$ and $\nu$ and $\eojhatetakn$ as an estimate of the entropic optimal joining cost.
Note that we will regard $\eta$ as a fixed quantity in all but Proposition \ref{prop:convergence_in_eta} and thus omit it in our notation, writing $\esteojkn$ for $\esteojetakn$ and $\eojhatkn$ for $\eojhatetakn$ whenever there is no risk of confusion.

As in the unregularized case, we establish in Appendix \ref{app:estimate_properties} that the estimate $\esteojkn$ is a joining of empirical estimates of $\mu$ and $\nu$ with the desired expected entropic cost $\eojhatkn$.

\subsection{Consistency and Error Bound}\label{sec:entropic_consistency}

If one is interested in consistency of the proposed estimates as $n$ goes to infinity, reasoning similar to that
in Section \ref{sec:estimate_oj} suggests that on consider sequences 
$\{\esteojkn\}_{n\geq 1}$ and $\{\eojhatkn\}_{n \geq 1}$ for appropriate sequences $k(n) \rightarrow \infty$.
The following result extends Theorems \ref{thm:consistent_estimation_of_oj} and \ref{thm:abstract_rates} to the regularized setting.

\begin{restatable}[]{thm}{consistentestimationofeoj}
\label{thm:consistent_estimation_of_eoj}
Let $\X$ and $\Y$ be finite and $\mu \in \calM_s(\X^\bbN)$ and $\nu \in \calM_s(\Y^\bbN)$ be ergodic.
Then for any $\eta > 0$, there exists a sequence $k = k(n)$ with $k \rightarrow \infty$ such that 
$\eojhatkn \rightarrow \eoj(\mu, \nu)$ and $\esteojkn \Rightarrow \Jmin^\eta(\mu, \nu)$ almost surely as $n\rightarrow \infty$.
\end{restatable}

The proof of Theorem \ref{thm:consistent_estimation_of_eoj} is similar to that of 
Theorem \ref{thm:consistent_estimation_of_oj}: we construct a $c$-admissible sequence $k(n)$ for $\mu$ and $\nu$ and apply a Lipschitz property of the entropic optimal transport cost to obtain the desired convergence.
An analog of Proposition \ref{prop:admissibility_consistency} holds for the regularized estimation scheme above.
The proofs of Theorem \ref{thm:abstract_rates} and Corollary \ref{cor:rates} can be extended to obtain an error 
bound for the estimated entropic optimal joining cost.

\begin{restatable}[]{thm}{eojerrorbound}
\label{thm:eoj_error_bound}
Let $\mu$ and $\nu$ have $\phi$-mixing coefficients $\phi_\mu$ and $\phi_\nu$, respectively, satisfying 
\begin{equation*}
\sum_{\ell=0}^n (n-\ell) \phi_\mu^{1/2}(\ell) = \calO(n^p) \quad\quad\mbox{and}\quad\quad \sum_{\ell=0}^n (n- \ell) \phi_\nu^{1/2}(\ell) = \calO(n^p)
\end{equation*}
for some $p \in [1, 2)$.
Then there exists a constant $C < \infty$ depending only on $\phi_\mu$ and $\phi_\nu$ such that for every $\eta > 0$, $k \geq 1$, $g \geq 0$, and large enough $n$,
\begin{align*}
\mathbb{E}\left|\eojhatkn - \eoj(\mu, \nu)\right| 
& \ \leq \ 
\|c\|_\infty (\phi_\mu(g+1) + \phi_\nu(g+1))\frac{k}{k+g} + (3\|c\|_\infty + 2\eta(\log |\X| + \log |\Y|)) \frac{g}{k} \\[.08in]
&\quad + u(k, n) \left(\frac{\|c\|_\infty}{2} + \frac{\eta}{k} \log \left(\frac{|\mathcal{X}|^{3k}}{u(k, n)}\right)\right) + v(k, n) \left(\frac{\|c\|_\infty}{2} + \frac{\eta}{k} \log \left(\frac{|\mathcal{Y}|^{3k}}{v(k, n)}\right)\right), 
\end{align*}
\vskip.03in
\noindent
where $u(k, n) = C |\X|^{k/2} n^{\nicefrac{p}{2}-1}$ and $v(k,n) = C |\Y|^{k/2} n^{\nicefrac{p}{2}-1}$.
In particular, if $k(n) < \frac{(2-p) \log n}{\log(|\X| \vee |\Y|) \vee 1}$ and $g(n) = o(k(n))$ with $k(n), g(n) \rightarrow \infty$, then the upper bound converges to zero as $n \rightarrow\infty$.
\end{restatable}

The introduction of entropy rate regularization into the optimal joining problem results in an error bound that is strictly worse than that for the unregularized optimal joining problem.
This reflects the increased difficulty of estimating an entropic optimal joining, which entails simultaneously learning the finite-dimensional distributions as well as the entropies of the marginal processes from observations.
Despite the worse error bound, we find that the sequence $k(n)$ may be chosen in the same way as in the unregularized case.
Moreover, for non iid processes (for which $k, g \rightarrow \infty$), one finds that the bound is of the same order as in the unregularized case.
For example, in the setting of Example \ref{ex:markov_bound}, we have $\mathbb{E}|\eojhatkn - \eoj(\mu, \nu)| = \calO\left(\frac{\log(\log(n))}{\log(n)}\right)$.
Thus from an asymptotic perspective, there is no additional price paid for using entropic regularization when estimating the optimal joining cost.

\section{Discussion}\label{sec:discussion}
The extension of optimal transport techniques to stochastic processes is an important problem in statistics and machine learning.
In this paper, we presented a step in this direction, considering the case of finite-alphabet, stationary and ergodic processes.
We argued that, in this setting, one should consider a constrained form of the optimal transport problem, referred to as the optimal joining problem, in order to account for the long-term dynamics of the processes of interest.
Given finite sequences of observations, we proposed estimates of an optimal joining and the optimal joining cost, and we proved that these estimates are consistent in the large sample limit.
We presented an upper bound on the expected error of the estimated optimal joining cost in terms of the mixing coefficients of the two processes of interest.
Finally, building upon recent work in optimal transport, we also proposed a regularized problem, the entropic optimal joining problem, and extended the proposed estimation scheme, consistency result, and error bound to this new problem.

This work enables the principled application of optimal transport techniques to data arising as observations from stationary processes.
Future work may investigate additional properties and uses of the entropic optimal joining problem.
For example, are there conditions under which the entropic optimal joining cost exhibits a faster rate of convergence compared to the unregularized optimal joining cost?
Other work may extend our results to the setting of Polish spaces.
It was noted in Section \ref{sec:estimate_oj} that the arguments in the proof of Theorem \ref{thm:consistent_estimation_of_oj} may be adapted to the case when $\X$ and $\Y$ are compact and $c$ is continuous.
However, it is not clear whether the entropic optimal joining is always well-defined in that setting.
Moreover, the arguments in the proof of Theorem \ref{thm:abstract_rates} do not extend easily to continuous spaces, and so further consideration is necessary.

\section{Proofs}\label{sec:proofs}

In this section, we prove our stated results.
We begin by proving Proposition \ref{prop:oj_is_otcbar} from Section \ref{sec:preliminaries}.
Next, we state and prove an inequality for the estimated optimal joining cost that will be used throughout the rest of this section.
Then, we prove our main results from Sections \ref{sec:estimate_oj}, \ref{sec:rates}, and \ref{sec:entropic_oj} in the order that they appear in the text.
A small selection of auxiliary results are proven in the appendix and will be referred to throughout this section.

\subsection{Proofs from Section \ref{sec:preliminaries}}

Here we establish the second equality in Proposition \ref{prop:oj_is_otcbar}, 
which states that the optimal joining cost is equal to the optimal transport cost with respect to the averaged cost $\overline{c}$.
We remind the reader that the first equality in Proposition \ref{prop:oj_is_otcbar} was established in \cite{gray1975generalization}.
We begin by showing that solutions to the optimal joining problem are characterized by a \emph{cyclical monotonicity}
property.

\begin{defn}
For two sets $\calU$ and $\calV$ and a cost function $c: \calU \times \calV \rightarrow \mathbb{R}$, a set $C \subset \calU \times \calV$ is called $c$-\emph{cyclically monotone} if for every $N \geq 1$ and every sequence $(u^1, v^1), ..., (u^N, v^N) \in C$,
\begin{equation*}
\sum\limits_{\ell = 1}^N c(u^\ell, v^\ell) \leq \sum\limits_{\ell = 1}^N c(u^\ell, v^{\ell + 1}),
\end{equation*}

\noindent with the convention that $v^{N+1} = v^1$.
A probability measure $\gamma$ on $\calU \times \calV$ is called $c$-\emph{cyclically monotone} if there exists a $c$-cyclically monotone set $C \subset \calU \times \calV$ such that $\gamma(C) = 1$.
\end{defn}

The characterization of optimal couplings in terms of cyclical monotonicity has been studied in the optimal transport literature.
We require the following result.

\begin{ethm}[\cite{beiglbock2015cyclical}]\label{thm:cycmon_implies_optimal}
Let $\calU$ and $\calV$ be Polish, $\mu \in \calM(\calU)$, $\nu \in \calM(\calV)$, and $c: \calU\times \calV\rightarrow [0, \infty)$ be measurable.
Then any $c$-cyclically monotone coupling $\pi \in \Pi(\mu, \nu)$ satisfying $\int c \, d\pi < \infty$ is a solution to $\otc(\mu, \nu)$.
\end{ethm}

\noindent Under stronger assumptions on $c$ (lower semicontinuity and integrability), 
one may also establish the reverse implication, namely that any optimal coupling is $c$-cyclically 
monotone (see \cite{villani2008optimal}).
An analogous result holds for the optimal joining problem.

\begin{lem}\label{lemma:cycmon_iff_optimal}
Let $\X$ and $\Y$ be finite and $\mu \in \calM_s(\X^\bbN)$ and $\nu \in \calM_s(\Y^\bbN)$ be ergodic.
Then an ergodic joining $\lambda \in \J(\mu, \nu)$ is a solution to $\oj(\mu, \nu)$ if and only if it is $\overline{c}$-cyclically monotone.
\end{lem}

\begin{proof}
The limiting average cost $\overline{c}$ is invariant under the joint left-shift map $\sigma \times \tau$, and therefore
by the pointwise ergodic theorem 
$\int c \, d\lambda_1 = \int \overline{c} \, d\lambda$ for every $\lambda \in \J(\mu, \nu)$.
Taking infima, we find that $\oj(\mu, \nu) = \ojcbar(\mu, \nu)$.
Let $\lambda \in \J(\mu, \nu)$ be an ergodic $\overline{c}$-cyclically monotone joining. 
Then $\lambda$ is a solution to $\otcbar(\mu, \nu)$ by Theorem \ref{thm:cycmon_implies_optimal}, and therefore
\begin{equation*}
\int c \, d\lambda_1 = \int \overline{c} \, d\lambda = \otcbar(\mu, \nu) \leq \ojcbar(\mu, \nu) = \oj(\mu, \nu)
\end{equation*}

\noindent and it follows that $\lambda$ is necessarily a solution to $\oj(\mu, \nu)$.

We now show that any ergodic optimal joining is $\overline{c}$-cyclically monotone.
Let $\lambda \in \Jmin(\mu, \nu)$ be ergodic.
As $\lambda$, $\mu$, and $\nu$ are ergodic, the pointwise ergodic theorem ensures that

\begin{enumerate}
\item There exists a set $D \subset \X^\bbN \times \Y^\bbN$ with $\lambda(D) = 1$ on which $\overline{c}$ is constant and equal to 
$\int c \, d\lambda_1$.
\item There exist sets $E \subset \X^\bbN$ and $F \subset \Y^\bbN$ such that $\mu(E) = \nu(F) = 1$, 
and for any $\bfx \in E$, $\bfy \in F$, the probability measures $\mu_{\mathbf{x}}^n := \frac{1}{n} \sum_{\ell = 0}^{n-1} \delta_{\sigma^\ell \bfx}$ and $\nu^n_{\mathbf{y}} := \frac{1}{n} \sum_{\ell=0}^{n-1} \delta_{\tau^\ell \bfy}$ satisfy $\mu_{\mathbf{x}}^n \Rightarrow \mu$ and $\nu_{\mathbf{y}}^n \Rightarrow \nu$.
\end{enumerate}

\noindent 
Let $C = D \cap (E \times F)$.  Then $\lambda(C) = 1$, 
so we need only show that $C$ is $\overline{c}$-cyclically monotone.
Let $N \geq 1$ and $(\bfx^1, \bfy^1), ..., (\bfx^N, \bfy^N) \in C$, and suppose by way of contradiction that 
\begin{equation*}
\sum\limits_{\ell=1}^N \overline{c}(\bfx^\ell, \bfy^\ell) > \sum\limits_{\ell=1}^N \overline{c}(\bfx^\ell, \bfy^{\ell+1}),
\end{equation*}

\noindent 
where we use the convention $\bfy^{N+1} = \bfy^1$.  Define a sequence of probability measures $\lambda^n$ on 
$\X^\bbN \times \Y^\bbN$ as follows
\begin{equation*}
\lambda^n := \frac{1}{nN}\sum\limits_{\ell=1}^N \sum\limits_{k=0}^{n-1} \delta_{(\sigma^k \bfx^\ell, \tau^k \bfy^{\ell+1})},
\end{equation*}

\noindent 
For each $n$ the measure $\lambda^n$ is a coupling of $\overline{\mu}^n = \frac{1}{N} \sum_{\ell=1}^N \mu^n_{\mathbf{x}^\ell}$ 
and $\overline{\nu}^n = \frac{1}{N} \sum_{\ell=1}^N \nu^n_{\mathbf{y}^\ell}$ 
where $\mu_{\mathbf{x}^\ell}^n$ and $\nu_{\mathbf{y}^\ell}^n$ are defined as in the definitions of the events $E$ and $F$ above.
The definition of $C$ ensures that $\overline{\mu}^n \Rightarrow \mu$ and $\overline{\nu}^n \Rightarrow \nu$ as $n\rightarrow \infty$.
Applying Lemma \ref{lemma:weakly_convergent_subsequence} in Appendix \ref{app:weakconvergence} we find that there is a subsequence $\lambda^{n_k}$ converging weakly to some $\tilde{\lambda} \in \J(\mu, \nu)$.
To simplify notation, we drop the subscript and refer to this subsequence as $\lambda^n$.
Using the fact that $\overline{c}$ is constant on $C$ and that $c$ is continuous and bounded, we have
\begin{align*}
\int c \, d\tilde{\lambda}_1 &=  \lim\limits_{n\rightarrow\infty} \int c \, d\lambda_1^n \\
&= \limsup\limits_{n\rightarrow\infty} \frac{1}{nN}\sum\limits_{\ell=1}^N \sum\limits_{k=1}^n c(x^\ell_k, y^{\ell+1}_k) \\
&\leq \frac{1}{N}\sum\limits_{\ell=1}^N \limsup\limits_{n\rightarrow\infty} \frac{1}{n}\sum\limits_{k=1}^n c(x^\ell_k, y^{\ell+1}_k) \\
&=  \frac{1}{N}\sum\limits_{\ell=1}^N \overline{c}(\bfx^\ell, \bfy^{\ell+1}) \\
&< \frac{1}{N} \sum\limits_{\ell=1}^N \overline{c}(\bfx^\ell, \bfy^\ell) \\
&= \int c \, d\lambda_1 \\
&= \oj(\mu, \nu),
\end{align*}
\noindent a contradiction. 
Thus $C$ is $\overline{c}$-cyclically monotone and the result follows.
\end{proof}

\ojisotcbar*
\begin{proof}
Under the stated conditions, it is known \citep{shields1996ergodic} that there exists an ergodic joining 
$\lambda \in \Jmin(\mu, \nu)$.
Lemma \ref{lemma:cycmon_iff_optimal} ensures that $\lambda$ is $\overline{c}$-cyclically monotone, and therefore,
by Theorem \ref{thm:cycmon_implies_optimal}, $\lambda$ is a solution to $\otcbar(\mu, \nu)$. 
It follows that $\oj(\mu, \nu) = \int c \, d\lambda_1 = \otcbar(\mu, \nu)$.
\end{proof}

\subsection{Preliminary Results}
Before proving the main results from Sections \ref{sec:estimate_oj}-\ref{sec:entropic_oj}, we establish a bound on the discrepancy between two entropic optimal transport costs in Lemma \ref{lemma:eot_is_lipschitz} below.
Let $\calU$ and $\calV$ be finite, $\alpha \in \calM(\calU)$ and $\beta \in \calM(\calV)$, $c: \calU \times \calV \rightarrow \mathbb{R}_+$, and $\eta > 0$.
Recall that the optimal transport cost satisfies 
\begin{equation}\label{eq:ot_dual}
\otc(\alpha, \beta) = \max\limits_{\substack{f: \calU \rightarrow \mathbb{R} \\ g: \calV \rightarrow \mathbb{R}}} \left\{ \int f \, d\alpha + \int g \, d\beta: f(u) + g(v) \leq c(u, v), \, \forall (u, v) \in \calU \times \calV\right\}.
\end{equation}
This equivalence is known as Kantorovich duality and is detailed, for example, in \cite{villani2008optimal}.
Moreover, it is established in \cite[Proposition 2.4]{cuturi2018semidual} that the entropic optimal transport problem satisfies
\begin{equation}\label{eq:eot_semidual}
\eotc(\alpha, \beta) = \max\limits_{f: \calU \rightarrow \mathbb{R}} \left\{\int f \, d\alpha + \int \tfb \, d\beta\right\} = \max\limits_{g: \calV \rightarrow \mathbb{R}} \left\{\int \tga \, d\alpha + \int g \, d\beta\right\},
\end{equation}
where 
\begin{equation*}
\tfb(v) = \eta \log \beta(v) - \eta \log \left(\sum\limits_u \exp\left\{\frac{1}{\eta}(f(u) - c(u,v))\right\} \right)
\end{equation*}
and
\begin{equation*}
\tga(u) = \eta \log \alpha(u) - \eta \log \left(\sum\limits_v \exp\left\{\frac{1}{\eta}(g(v) - c(u,v))\right\} \right).
\end{equation*}
\noindent The formulation \eqref{eq:eot_semidual} is referred to as the semidual of the entropic optimal transport problem while the quantities $\tfb$ and $\tga$ are referred to as the $(c,\eta)$-transforms of $f$ and $g$ with respect to $\beta$ and $\alpha$, respectively.
In what follows, we will let
\begin{equation*}
\tg(u) =- \eta \log \left(\sum\limits_v \exp\left\{\frac{1}{\eta}(g(v) - c(u,v))\right\} \right),
\end{equation*}
\noindent to simplify notation.
Note that $\tga(u) = \eta \log \alpha(u) + \tg(u)$.
Our proof of Lemma \ref{lemma:eot_is_lipschitz} will leverage the duality \eqref{eq:ot_dual} and \eqref{eq:eot_semidual} 
as well as the following basic facts about $\tf$ and $\tg$.

\begin{restatable}[]{lem}{cetatransformfacts}
\label{lem:ceta_transform_facts}
Let $(\calU, d_\calU)$ and $(\calV, d_\calV)$ be finite pseudometric spaces, and let $f: \calU \rightarrow \mathbb{R}$ and $g: \calV \rightarrow \mathbb{R}$ be real-valued functions.
Furthermore, let $c: \calU \times \calV \rightarrow \mathbb{R}_+$ be a non-negative cost function satisfying $|c(u, v) - c(u', v')| \leq M (d_\calU(u, u') + d_\calV(v, v'))$ for all $u, u' \in \calU$ and $v, v' \in \calV$ for some $M \in \mathbb{R}$.
Then for any $\eta > 0$, $\tf$ and $\tg$ satisfy $|\tf(v) - \tf(v')| \leq M d_\calV(v, v')$ and $|\tg(u) - \tg(u')| \leq M \, d_\calU(u, u')$ for all $u, u' \in \calU$ and $v, v' \in \calV$.
\end{restatable}

\noindent 
A proof of Lemma \ref{lem:ceta_transform_facts} is provided in Appendix \ref{sec:transform_properties}.
A detailed discussion of the $(c, \eta)$-transform and its use in optimal transport can be found in \cite{peyre2019computational}.
Now we may proceed to the result of interest.

\begin{lem}\label{lemma:eot_is_lipschitz}
Let $\calU$ and $\calV$ be finite and let $c : \calU\times\calV \rightarrow \mathbb{R}_+$ be a cost function.
Then for any $\eta \geq 0$, any $\alpha, \alpha' \in \calM(\calU)$, and any $\beta \in \calM(\calV)$,
it holds that
\begin{align}\label{eq:eot_is_lipschitz}
\eotc(\alpha, \beta) - \eotc(\alpha', \beta) \leq \otcU(\alpha, \alpha') + \eta (H(\alpha') - H(\alpha)).
\end{align}
The analogous bound for a pair of measures in $\calM(\calV)$ also holds.
\end{lem}

\begin{proof}
We begin by considering the case $\eta > 0$.
Let $\alpha$, $\alpha' \in \calM(\calU)$ and let
$\beta \in \calM(\calV)$.
By \eqref{eq:eot_semidual} there exists $g: \calV \rightarrow \mathbb{R}$ be such that 
\begin{equation}\label{eq:dual_upperbound}
\eotc(\alpha, \beta) = \int \tga \, d\alpha + \int g \, d\beta.
\end{equation}
\noindent Rewriting $\tga$, we have
\begin{align*}
\eotc(\alpha,\beta) &= \int \left(\tg + \eta \log \alpha\right) \, d\alpha + \int g \, d\beta \\
&= \int \tg \, d\alpha + \int g \, d\beta - \eta H(\alpha).
\end{align*}
\noindent As $g$ is feasible for the semidual problem of $\eotc(\alpha', \beta)$, \eqref{eq:eot_semidual} implies that
\begin{align}
\eotc(\alpha', \beta) &\geq \int \tgap \, d\alpha' + \int g \, d\beta \nonumber \\
&= \int \tg \, d\alpha' + \int g \, d\beta - \eta H(\alpha').\label{eq:semidual_lb}
\end{align}
\noindent Combining \eqref{eq:dual_upperbound} and \eqref{eq:semidual_lb}, we find that
\begin{equation*}
\eotc(\alpha, \beta) - \eotc(\alpha', \beta) \leq  \int \tg \, d\alpha - \int \tg \, d\alpha'  + \eta (H(\alpha') - H(\alpha)).
\end{equation*}
\noindent Note that  
$|c(u, v) - c(\tilde{u}, \tilde{v})| \leq \cU(u, \tilde{u}) + \cV(v, \tilde{v})$, and therefore
by Lemma \ref{lem:ceta_transform_facts}, $\tg$ satisfies $\tg(u) - \tg(\tilde{u}) \leq \cU(u, \tilde{u})$.
Thus the pair $(\tg, -\tg)$ is feasible for the dual \eqref{eq:ot_dual} of $\otcU(\alpha, \alpha')$ and it follows that
\begin{equation*}
\eotc(\alpha, \beta) - \eotc(\alpha', \beta) \leq \otcU(\alpha, \alpha') + \eta (H(\alpha') - H(\alpha)).
\end{equation*}
Taking the limit as $\eta \rightarrow 0$ of \eqref{eq:eot_is_lipschitz} and applying \cite[Proposition 2.1]{cuturi2018semidual}, we obtain the result for $\eta = 0$.
\end{proof}

The next proposition details the implication of Lemma \ref{lemma:eot_is_lipschitz}	
for the $k$-step entropic optimal transport cost.
The proof follows from a straightforward application of Lemma \ref{lemma:eot_is_lipschitz} and the pointwise inequalities $\ckXk \leq \cXk$ and $\ckYk \leq \cYk$.

\begin{prop}\label{prop:eot_is_lipschitz}
For any $\eta \geq 0$, $n \geq 1$, and $k \in \{1, ..., n\}$, 
\begin{align*}
\left|\eojhatk(X_1^n, Y_1^n) - \frac{1}{k} \eotck(\mu_k, \nu_k)\right| &\leq \frac{1}{k} \otcXk(\hat{\mu}_{k,n}, \mu_k) + \frac{\eta}{k} \left|H(\mu_k) - H(\hat{\mu}_{k,n})\right| \\
&\quad\quad + \frac{1}{k} \otcYk(\hat{\nu}_{k,n}, \nu_k) + \frac{\eta}{k} \left|H(\nu_k) - H(\hat{\nu}_{k,n})\right|.
\end{align*}
\end{prop}

\subsection{Proofs from Section \ref{sec:estimate_oj}}
In this section, we prove Theorem \ref{thm:consistent_estimation_of_oj} regarding the consistency of the proposed estimates without entropic regularization.

\consistentestimationofoj*
\begin{proof}
We begin by constructing a sequence $\{k(n)\}$ such that the $k(n)$-step empirical optimal transport cost converges to the optimal joining cost almost surely.
As noted by \cite{marton1994entropy}, due to the ergodic theorem, $\mu$ and $\nu$ have admissible sequences $\{\ell(n)\}$ and $\{m(n)\}$.
Using the same reasoning, one may verify that $\{k(n)\}$ where $k(n) = \min\{\ell(n), m(n)\}$ is also admissible for both processes.
Since any admissible sequence is also $c$-admissible, $\{k(n)\}$ is $c$-admissible for both $\mu$ and $\nu$.
In order to simplify notation in the rest of the proof, we suppress the dependence of $k(n)$ on $n$.
For any $n \geq 1$, an application of the triangle inequality gives
\begin{equation}\label{eq:unreg_cost_triangle_ineq}
\left|\ojhatk(X_1^n, Y_1^n) -  \oj(\mu, \nu)\right| \leq \left|\ojhatk(X_1^n, Y_1^n) - \frac{1}{k} \otck\left(\mu_k, \nu_k\right)\right| + \left|\frac{1}{k}\otck\left(\mu_k, \nu_k\right) - \oj(\mu, \nu)\right|.
\end{equation}

\noindent Applying Proposition \ref{prop:eot_is_lipschitz}, we have
\begin{equation*}
\left|\ojhatk(X_1^n, Y_1^n) - \frac{1}{k} \otck(\mu_{k}, \nu_{k})\right| \leq \frac{1}{k} \otcXk(\hat{\mu}_k[X_1^n], \mu_{k}) + \frac{1}{k} \otcYk(\hat{\nu}_k[Y_1^n], \nu_{k}),
\end{equation*}

\noindent which by the $c$-admissibility of $k$ for $\mu$ and $\nu$ implies that the first term on the right hand side in \eqref{eq:unreg_cost_triangle_ineq} goes to zero, almost surely as $n \rightarrow\infty$.
An application of Proposition \ref{prop:oj_is_otcbar} with the fact that $k(n) \rightarrow\infty$ shows that the second term on the right hand side in \eqref{eq:unreg_cost_triangle_ineq} goes to zero.
It follows that
\begin{equation}\label{eq:unreg_cost_consistency}
\left|\ojhatk(X_1^n, Y_1^n) -  \oj(\mu, \nu)\right| \rightarrow 0, \quad \mbox{almost surely}.
\end{equation}

Next we show that the sequence of estimated optimal joinings indexed by $k$ converges weakly to the set of optimal joinings $\Jmin(\mu, \nu)$, almost surely.
Fix sequences $X = X_1, X_2, ...$ and $Y = Y_1, Y_2, ...$ in sets of $\mu$- and $\nu$-measure one on which \eqref{eq:unreg_cost_consistency} holds.
Let $\{\hat{\lambda}^{k, n}\}_{n\geq 1}$ be the corresponding sequence of estimated optimal joinings. 
By Lemma \ref{lemma:weakly_convergent_subsequence} in Appendix \ref{app:weakconvergence}, for any subsequence $\{\hat{\lambda}^{k, n_\ell}\}_{\ell \geq 1}$, there is a further subsequence converging weakly to a joining $\lambda \in \J(\mu, \nu)$.
For ease of notation, we refer to this further subsequence again as $\{\hat{\lambda}^{k, n_\ell}\}_{\ell \geq 1}$.
Then 
\begin{equation*}
\int\limits c \, d\lambda_1 = \lim\limits_{\ell \rightarrow\infty} \int c \, d\hat{\lambda}^{k, n_\ell}_1 =\lim\limits_{\ell \rightarrow\infty} \ojhatk(X_1^{n_\ell}, Y_1^{n_\ell}) = \oj(\mu, \nu),
\end{equation*}

\noindent where the first equality follows from the continuity and boundedness of $c$, the second equality follows from Proposition \ref{prop:entropic_est_is_joining} in Appendix \ref{app:estimate_properties}, and the third equality follows from \eqref{eq:unreg_cost_consistency}.
Thus, $\lambda \in \Jmin(\mu, \nu)$ and since the subsequence was arbitrary, we conclude that $\hat{\lambda}^{k, n_\ell} \Rightarrow \Jmin(\mu, \nu)$.
By the choice of sequences $X$ and $Y$, this convergence occurs almost surely.
\end{proof}

\subsection{Proofs from Section \ref{sec:rates}}
In this section, we prove Theorem \ref{thm:abstract_rates} regarding the expected error of the estimated optimal joining cost $\ojhatkn$.
Our argument may be broken down into three steps. First, we prove a Lipschitz result for the optimal joining cost akin to Lemma \ref{lemma:eot_is_lipschitz} in terms of Ornstein's $\overline{d}$-distance.
Second, we prove a novel upper bound on these $\overline{d}$ terms using the $\phi$-mixing coefficients of the process measures $\mu$ and $\nu$.
Finally, we use a covering number bound to control the error of the estimated $k$-step optimal transport cost.

\paragraph{Bound on Optimal Joining Cost Discrepancy.}
To begin, we establish an inequality for the difference between the entropic optimal joining costs of two pairs of processes akin to the result stated in Lemma \ref{lemma:eot_is_lipschitz}.
Since we will require the bound for the proof of Theorem \ref{thm:consistent_estimation_of_eoj} as well, we prove Lemma \ref{lemma:oj_is_lipschitz} more generally for the regularized optimal joining cost $\eoj(\mu, \nu)$.
Briefly, we remind the reader that the $\overline{d}$-distance between two processes, introduced in \cite{ornstein1973application}, may be defined as the optimal joining cost with respect to the single-letter Hamming metric $(\bfu, \bfu') \mapsto \mathbbm{1}(u_1 \neq u'_1)$.
The distance $\overline{d}$ may be thought of as the process analogue to the total variation distance.

\begin{lem}\label{lemma:oj_is_lipschitz}
Let $\alpha, \alpha' \in \calM_s(\X^\bbN)$ be stationary process measures.
Then for any $\eta \geq 0$ and $\beta \in \calM_s(\Y^\bbN)$,
\begin{equation*}
\eoj(\alpha, \beta) - \eoj(\alpha', \beta) \leq \|c\|_\infty \overline{d}(\alpha, \alpha') + \eta (h(\alpha') - h(\alpha)).
\end{equation*}
The analogous bound for stationary process measures in $\calM_s(\Y^\bbN)$ also holds.
\end{lem}
\begin{proof}
Fix $k \geq 1$, $\eta \geq 0$, $\alpha$, $\alpha' \in \calM_s(\X^\bbN)$, and $\beta \in \calM_s(\Y^\bbN)$.
Recall that by Lemma \ref{lemma:eot_is_lipschitz}, 
\begin{equation*}
\eotck\left(\alpha_k, \beta_k\right) - \eotck\left(\alpha'_k, \beta_k\right) \leq \ot_{\ckXk}(\alpha_k, \alpha'_k) + \eta (H(\alpha'_k) - H(\alpha_k)).
\end{equation*}
\noindent Let $\delta_k(x_1^k, \tilde{x}_1^k) = \sum_{\ell=1}^k \mathbbm{1}(x_\ell \neq \tilde{x}_\ell)$ be the $k$-step Hamming distance and note that the pointwise upper bound $\ckXk \leq \|c\|_\infty \delta_k$ holds.
Thus,
\begin{equation*}
\eotck\left(\alpha_k, \beta_k\right) - \eotck\left(\alpha'_k, \beta_k\right) \leq \|c\|_\infty \ot_{\delta_k}(\alpha_k, \alpha'_k) + \eta (H(\alpha'_k) - H(\alpha_k)).
\end{equation*}
Dividing by $k$, letting $k \rightarrow \infty$, and applying Proposition \ref{prop:convergence_of_eot_to_eoj}, we find
\begin{equation*}
\eoj(\alpha, \beta) - \eoj(\alpha', \beta) \leq \|c\|_\infty \oj_\delta(\alpha, \alpha') + \eta (h(\alpha') - h(\alpha)).
\end{equation*}
\noindent Recognizing that $\oj_\delta(\alpha, \alpha') = \overline{d}(\alpha, \alpha')$, the result follows.
\end{proof}

\paragraph{Bound on $\overline{d}$.}
Next, we prove an upper bound on the $\overline{d}$-distance between a stationary process measure and an approximation constructed from its finite dimensional distributions.
The approximation of interest is defined as follows:

\begin{defn}[Block approximation with gaps]
Let $\calU$ be a finite space and $k, g \geq 1$.
We define $\tilde{\Lambda}^k: \calM_s(\calU^\bbN) \times \calM(\calU^g) \rightarrow \calM(\calU^\bbN)$ to be the map that takes a process $\gamma \in \calM_s(\calU^\bbN)$ and a probability measure $\alpha \in \calM(\calU^g)$ to the unique probability measure on $\calU^\bbN$ obtained by independently concatenating $\gamma_k$ and $\alpha$ together infinitely many times.
Formally, for any $\ell(k+g)$-dimensional cylinder set $C \subset \calU^\bbN$,
\begin{equation*}
\tilde{\Lambda}^k[\gamma, \alpha](C) = \prod\limits_{i=0}^{\ell-1} \gamma_k(C_{i(k+g)+1}^{i(k+g)+k}) \alpha(C_{i(k+g)+k+1}^{(i+1)(k+g)}).
\end{equation*}
Moreover, we define $\Lambda^k : \calM_s(\calU^\bbN) \times \calM(\calU^g) \rightarrow \calM_s(\calU^\bbN)$ to be the map defined by randomizing the start of the output of $\tilde{\Lambda}^k$ over the first $k+g$ coordinates.
Formally, for any set $U \subset \calU^\bbN$,
\begin{equation*}
\Lambda^k[\gamma, \alpha](U) = \frac{1}{k+g} \sum\limits_{\ell=0}^{k+g-1} \tilde{\Lambda}^k[\gamma, \alpha](\calU^\ell \times U).
\end{equation*}

\noindent We will refer to $\tilde{\Lambda}^k[\gamma, \alpha]$ as the \emph{independent $k$-block process approximation of $\gamma$ with gap $g$} and $\Lambda^k[\gamma, \alpha]$ as the \emph{stationary $k$-block process approximation of $\gamma$ with gap $g$}.
\end{defn}

Note the abuse of notation for the block approximation with gaps compared to the block approximations without gaps.
In the rest of the paper, it will be understood that we omit gaps when $\tilde{\Lambda}^k$ or $\Lambda^k$ take one argument and include gaps when either takes two arguments.
Note also that we omit $\alpha$ when referring to either approximation because our arguments do not depend on the choice of $\alpha$.
For simplicity, we will use the same notation for these approximations regardless of the alphabet of the processes under consideration.
As such, $\Lambda^k[\mu, \alpha]$ and $\Lambda^k[\nu, \beta]$ are well-defined.
We will show later that $\Lambda^k[\mu, \alpha]$ and $\Lambda^k[\nu, \beta]$ arise naturally in the proof of Theorem \ref{thm:abstract_rates}.
In particular, it will be necessary to control the error of these approximations as measured by the $\overline{d}$-distances to $\mu$ and $\nu$, respectively.

\begin{lem}\label{lemma:d_bar_rates}
Let $\calU$ be finite and $\gamma \in \calM_s(\calU^\bbN)$ have $\phi$-mixing coefficient $\phi_\gamma$.
Then for every $k \geq 1$, $g \geq 0$ and $\alpha \in \calM(\calU^g)$,
\begin{equation*}
\overline{d}(\gamma, \Lambda^k[\gamma, \alpha]) \leq \frac{g}{k+g} + \frac{k}{k+g} \phi_\gamma(g+1).
\end{equation*}
\end{lem}
\begin{proof}
Fix $k \geq 1$, $g \geq 0$, $\alpha \in \calM(\calU^g)$.
To simplify notation, let $\tilde{\xi} = \tilde{\Lambda}^k[\gamma, \alpha]$ and $\xi = \Lambda^k[\gamma, \alpha]$.
We begin by defining an intermediate process $\zeta \in \calM_s(\calU^\bbN)$.
Let $\tilde{\zeta} \in \calM(\calU^\bbN)$ be the probability measure corresponding to the distribution of the process $V = V_1, V_2, \dots$ generated by drawing a sequence $U = U_1, U_2, ...$ according to $\gamma$ and replacing the $g$-blocks of $U$ with independent draws $(G_1^i, ..., G_g^i)$ from $\alpha$ to obtain
\begin{equation}
\label{eq:Utildeseq}
V = \underbrace{U_1, \dots, U_k}_{k-\mbox{\scriptsize block 1}}, \underbrace{G_1^1, \dots, G_g^1}_{g-\mbox{\scriptsize block 1}}, \underbrace{U_{k+g+1}, \dots, U_{2k+g}}_{k-\mbox{\scriptsize block 2}}, \underbrace{G_1^2, \dots, G_g^2}_{g-\mbox{\scriptsize block 2}}, \dots.
\end{equation}
Then, let $\zeta \in \calM_s(\calU^\bbN)$ be the stationary process measure obtained by randomizing the start of $\tilde{\zeta}$ over the first $k+g$ coordinates.
By the triangle inequality for $\overline{d}$,
\begin{equation*}
\overline{d}(\gamma, \xi) \leq \overline{d}(\gamma, \zeta) + \overline{d}(\zeta, \xi).
\end{equation*}

\noindent Since the $\overline{d}$-distance is defined as an infimum over joinings, we may upper bound both terms on the right hand side by the expected cost of some suitably chosen joinings of $\gamma$ and $\zeta$, and $\zeta$ and $\xi$, respectively.

We will first bound $\overline{d}(\gamma, \zeta)$ by constructing a coupling of $\gamma$ and $\tilde{\zeta}$ with low expected cost and then randomizing the start to obtain a joining of $\gamma$ and $\zeta$.
Since the $k$-blocks of $\tilde{\zeta}$ are equal in distribution to those of $\gamma$ by construction, we may couple them so that the $k$-blocks are equal with probability one.
Formally, define the coupling $\pi \in \Pi(\gamma, \tilde{\zeta})$ to be the probability measure corresponding to the distribution of the process $(U', V')$ generated by drawing the sequence $U' = U'_1, U'_2, \dots$ according to $\gamma$ and letting $V' = U'_1, \dots, U'_k, G^1_1, \dots G^1_g, U'_{k+g+1}, \dots$ as in \eqref{eq:Utildeseq}, replacing the $g$-blocks of $U'$ with independent draws $(G^i_1, \dots, G^i_g)$ from $\alpha$.
In particular, $U'_\ell = V'_\ell$ with $\pi$-probability one when $\ell = i(k+g)+j$ for some $i \geq 0$ and $j \in \{1, ..., k\}$.
Letting $\lambda \in \J(\gamma, \zeta)$ be the joining obtained by randomizing the start of $\pi$ over the first $k+g$ coordinates, we obtain
\begin{align*}
\overline{d}(\gamma, \zeta) &\leq \int \mathbbm{1}(u \neq v) \, d\lambda_1(u, v) \\
&= \int \left(\frac{1}{k+g} \sum\limits_{\ell=1}^{k+g} \mathbbm{1}(u_\ell \neq v_\ell)\right) \, d\pi_{k+g}(u_1^{k+g}, v_1^{k+g}) \\
&= \int \left(\frac{1}{k+g} \sum\limits_{\ell=k+1}^{k+g} \mathbbm{1}(u_\ell \neq v_\ell)\right) \, d\alpha(u_{k+1}^{k+g}, v_{k+1}^{k+g}) \\
& \leq \frac{g}{k+g}.
\end{align*}
Next we bound $\overline{d}(\zeta, \xi)$.
By Proposition \ref{prop:oj_is_otcbar},
\begin{equation*}
\overline{d}(\zeta, \xi) = \lim_{m\rightarrow\infty} \frac{1}{m} \ot_{\delta_m}\left(\zeta_m, \xi_m\right)
\end{equation*}

\noindent and thus, fixing a subsequence $m(\ell) := \ell(k+g) + k$ for $\ell \in \bbN_0$, we have
\begin{equation}\label{eq:dbar_is_limit_of_tv}
\overline{d}(\zeta, \xi) = \lim_{L\rightarrow\infty} \frac{1}{m(L)} \ot_{\delta_{m(L)}}\left(\zeta_{m(L)}, \xi_{m(L)}\right).
\end{equation}

\noindent It suffices to obtain a bound on 
\begin{equation*}
\frac{1}{m(L)}\ot_{\delta_{m(L)}}\left(\zeta_{m(L)}, \xi_{m(L)}\right)
\end{equation*}

\noindent for fixed $L \in \bbN$ and take a limit as $L\rightarrow \infty$.
Similar to the first bound, we will achieve this by constructing a coupling of $\tilde{\zeta}_{m(L)}$ and $\tilde{\xi}_{m(L)}$ with low expected cost and randomizing the start to obtain a coupling of $\zeta_{m(L)}$ and $\xi_{m(L)}$.
Fix $L \in \bbN$ and recall that both $\tilde{\zeta}$ and $\tilde{\xi}$ are comprised of alternating blocks of size $k$ and $g$, with the difference between the two measures being that the $k$-blocks of $\tilde{\zeta}$ depend upon one another while those of $\tilde{\xi}$ are independent of one another.
In order to obtain the desired bound, we will bridge the gap between $\tilde{\zeta}_{m(L)}$ and $\tilde{\xi}_{m(L)}$ with a series of intermediate process measures $\rho^0, ..., \rho^L \in \calM(\calU^\bbN)$ where the first $\ell+1$ $k$-blocks of $\rho^\ell$ are dependent on one another (as in $\tilde{\zeta}_{m(L)}$) and the rest are independent (as in $\tilde{\xi}_{m(L)}$).
In other words, $\rho^\ell$ describes the distribution of the process
\begin{equation*}
\underbrace{U_1, \dots, U_k}_{k-\mbox{\scriptsize block 1}}, \underbrace{G_1^1, \dots, G_g^1}_{g-\mbox{\scriptsize block}}, \dots, \underbrace{U_{m(\ell)-k+1}, \dots, U_{m(\ell)}}_{k-\mbox{\scriptsize block } \ell+1}, \underbrace{\dots}_{g-\mbox{\scriptsize block}}, \underbrace{\tU_1^1, \dots, \tU_k^1}_{k-\mbox{\scriptsize block } \ell+2}, \underbrace{\dots}_{g-\mbox{\scriptsize block}}, \underbrace{\tU_1^2, \dots, \tU_k^2}_{k-\mbox{\scriptsize block } \ell+3}, \dots
\end{equation*}
where the sequence $U_1, U_2, \dots$ is drawn according to $\gamma$ and for each $i \geq 1$, $(G_1^i, \dots, G_g^i)$ and $(\tU_1^i, ..., \tU_k^i)$ are independent draws from $\alpha$ and $\gamma_k$, respectively.
In other words, $\rho^\ell$ is equal to $\tilde{\zeta}$ on the first $\ell+1$ $(k+g)$-blocks and equal to $\tilde{\xi}$ on the remaining blocks.
Note that $\rho^0 = \tilde{\xi}$ and $\rho^L_{m(L)} = \tilde{\zeta}_{m(L)}$.
Applying the triangle inequality for the optimal transport cost,
\begin{equation}\label{eq:m_dim_ot_triangle}
\ot_{\delta_{m(L)}}\left(\tilde{\zeta}_{m(L)}, \tilde{\xi}_{m(L)}\right) \leq \sum\limits_{\ell=0}^{L-1} \ot_{\delta_{m(L)}}\left(\rho_{m(L)}^\ell, \rho_{m(L)}^{\ell+1}\right).
\end{equation}

In order to bound the terms on the right hand side of \eqref{eq:m_dim_ot_triangle}, we will couple each pair $\rho_{m(L)}^\ell$ and $\rho_{m(L)}^{\ell+1}$ so that they are equal on the first $\ell + 1$ $(k+g)$-blocks, close on the next $k$-block, and equal again on the remaining $k$- and $g$-blocks.
Fix $\ell \in \{0, ..., L-1\}$ and consider the coupling of $\rho^\ell_{m(L)}$ and $\rho^{\ell+1}_{m(L)}$ corresponding to the distribution of the paired process $(U', V') = (U'_1, V'_1)$, $\dots$, $(U'_{m(L)}, V'_{m(L)})$, defined as
\begin{align}
\label{eq:couplingseq}
\begin{split}
U' &=U_1, \dots, U_k, \,\dots\,, \dots, U_{m(\ell)-k+1}, \dots, U_{m(\ell)}, \,\dots,\, \tU'_1, \,\dots, \tU'_k, \,\dots,\, \tU_1^1, \dots, \tU_k^1, \dots    \\
V' &= \underbrace{U_1, \dots, U_k}_{k-\mbox{\scriptsize block 1}}, \!\!\!\underbrace{\dots}_{g-\mbox{\scriptsize block}}\!\!\!, \dots, \underbrace{U_{m(\ell)-k+1}, \dots, U_{m(\ell)}}_{k-\mbox{\scriptsize block } \ell+1}, \!\!\underbrace{\dots}_{g-\mbox{\scriptsize block}}\!\!, \underbrace{\tV'_1, \dots, \tV'_k}_{k-\mbox{\scriptsize block } \ell+2}, \!\!\underbrace{\dots}_{g-\mbox{\scriptsize block}}\!\!\!, \underbrace{\tU_1^1, \dots, \tU_k^1}_{k-\mbox{\scriptsize block } \ell+3}, \dots
\end{split}
\end{align}
with elements defined as follows:
The sequence $U_1, U_2, \cdots$ is drawn according to $\gamma$.
The $g$-blocks of $U'$ and $V'$, which are omitted from \eqref{eq:couplingseq}, are equal and drawn independently according to $\alpha$.
The $k$-blocks $(\tU_1^i, \dots, \tU_k^i)$ are drawn independently of one another according to the $k$-dimensional distribution $\gamma_k$ of $\gamma$.
And the paired sequence $(\tU'_1, \tV'_1)$, $\dots$, $(\tU'_k, \tV'_k)$ is drawn according to an optimal coupling of $\gamma_k$ and $\gamma_k(\cdot | [U_1^{m(\ell)}]_k)$ with respect to $\delta_k$, where $\gamma_k(\cdot | [U_1^{m(\ell)}]_k)  \in \calM(\calU^k)$ is the distribution of $U_{m(\ell+1)-k+1}, ..., U_{m(\ell+1)}$ conditioned on the previous $k$-blocks $(U_1, \dots, U_k)$, $\dots$, $(U_{m(\ell)-k+1}, \dots,$ $U_{m(\ell)})$. 

It is clear from \eqref{eq:couplingseq} that the coupling $(U', V')$ only incurs a cost at $k$-block $\ell+2$.
By construction, this cost is equal to the optimal transport cost of $\gamma_k$ and $\gamma_k(\cdot | [U_1^{m(\ell)}]_k)$ with respect to $\delta_k$.
Thus
\begin{equation*}
\ot_{\delta_{m(L)}}\left(\rho_{m(L)}^\ell, \rho_{m(L)}^{\ell+1}\right) \leq \bbE\left[\delta_{m(L)}(U', V')\right] = \bbE \left[\ot_{\delta_k}\left(\gamma_k, \gamma_k(\cdot | [U_1^{m(\ell)}]_k)\right)\right].
\end{equation*}
Finally, using the fact that for any $(u_1^k, v_1^k)$, $\delta_k(u_1^k, v_1^k) \leq k \delta(u_1^k, v_1^k)$, we have
\begin{equation*}
\ot_{\delta_{m(L)}}\left(\rho_{m(L)}^\ell, \rho_{m(L)}^{\ell+1}\right) \leq k \bbE \left[\ot_\delta\left(\gamma_k, \gamma_k(\cdot | [U_1^{m(\ell)}]_k)\right)\right] \leq k \max\limits_{u_1^{m(\ell)}} \ot_\delta\left(\gamma_k, \gamma_k(\cdot | [u_1^{m(\ell)}]_k)\right).
\end{equation*}
Finally, using the fact that the optimal transport cost with respect to $\delta$ is equal to the total variation distance, we have
\begin{equation*}
\ot_{\delta_{m(L)}}\left(\rho_{m(L)}^\ell, \rho_{m(L)}^{\ell+1}\right) \leq k \max\limits_{u_1^{m(\ell)}} \max\limits_{A \in \calU^k} \left|\gamma_k(A | [u_1^{m(\ell)}]_k) - \gamma_k(A)\right|.
\end{equation*}
Letting $\phi_\gamma: \bbN \rightarrow \mathbb{R}_+$ be the mixing coefficient of $\gamma$, it follows that
\begin{equation*}
\ot_{\delta_{m(L)}}\left(\rho_{m(L)}^\ell, \rho_{m(L)}^{\ell+1}\right) \leq k \phi_\gamma(g+1).
\end{equation*}
Plugging this result into \eqref{eq:m_dim_ot_triangle},
\begin{equation*}
\ot_{\delta_{m(L)}}\left(\tilde{\zeta}_{m(L)}, \tilde{\xi}_{m(L)}\right) \leq \sum\limits_{\ell=0}^{L-1} k \phi_\gamma(g+1) = Lk \phi_\gamma(g+1).
\end{equation*}

\noindent By randomizing the start of the couplings considered above, one may further establish that 
\begin{equation*}
\ot_{\delta_{m(L)}}\left(\zeta_{m(L)}, \xi_{m(L)}\right) \leq Lk \phi_\gamma(g+1).
\end{equation*}

\noindent Plugging this into \eqref{eq:dbar_is_limit_of_tv} and recalling that $m(L) = L(k+g)+k$, we find that
\begin{equation*}
\overline{d}(\zeta, \xi) = \lim\limits_{L \rightarrow\infty} \frac{Lk}{L(k+g)+k} \phi_\gamma(g+1) = \frac{k}{k+g} \phi_\gamma(g+1).
\end{equation*}

\noindent Combining this and the earlier bound yields the result.
\end{proof}

\paragraph{Bound on the Mean $k$-step Optimal Transport Cost.}
In the final step before proving Theorem \ref{thm:abstract_rates}, we prove an upper bound on the mean optimal transport cost between $\hat{\mu}_{k,n}$ and $\mu_k$ in terms of $\phi_\mu$ (and the analogous result for $\nu$).
In order to do this, we leverage \cite[Proposition 1.7]{boissard2014mean}, stated below as Theorem \ref{thm:rho_mixing_bound}, regarding the expectation of the $p$-Wasserstein distance from an empirical measure to its target measure for stationary, $\rho$-mixing sequences.
We will say that a process measure $\gamma \in \calM_s(\calU^\bbN)$ has $\rho$-mixing coefficient $\rho_\gamma: \bbN_0 \rightarrow \bbR_+$ if $\rho_\gamma(0) = 1$ and for $g > 1$ and any random variable $U = (U_1, U_2, ...): \Omega \rightarrow \calU^\bbN$ distributed according to $\gamma$,
\begin{equation*}
\rho_\gamma(g) := \sup\left\{|\mbox{Corr}(F, G)|: \ell \geq 1, F \in \mathcal{L}^2(\sigma(U_1, ..., U_\ell)), G \in \mathcal{L}^2(\sigma(U_{\ell+g}, ...))\right\},
\end{equation*}

\noindent where for $i \leq j \leq \infty$, $\sigma(U_i, ..., U_j)$ is the smallest sigma field in $(\Omega, \mathcal{B}, \mathbb{P})$ with respect to which $U_i^j$ is measurable and for a sigma field $\mathcal{F} \subset \mathcal{B}$, $\mathcal{L}^2(\mathcal{F})$ is the set of square-integrable, $\mathcal{F}$-measurable random variables.
The result is stated below in a form that is adapted to our notation and the case of $p = 1$.

\begin{ethm}[\cite{boissard2014mean}]\label{thm:rho_mixing_bound}
Let $\gamma \in \mathcal{M}_s(\calU^\bbN)$ be a stationary process measure on a Polish space $\calU$ with metric $d$ and let $\gamma$ have $\rho$-mixing coefficient $\rho_\gamma$.
Define $\chi_n = n^{-2} \sum_{m=0}^n \sum_{g=0}^m \rho_\gamma(g)$ and let $\Delta = \diam(\calU)$.
If $\gamma_1^n := \gamma_1^n[U_1^n] \in \mathcal{M}(\calU)$ is the empirical measure constructed from samples $U_1^n$ drawn according to $\gamma$, then there exists a constant $C < \infty$ such that for any $t \in (0, \nicefrac{\Delta}{4}]$,
\begin{equation*}
\mathbb{E}\ot_d(\gamma_1^n, \gamma_1) \leq C \left(t + \chi_n^{\nicefrac{1}{2}} \int_t^{\frac{1}{4}\Delta} \mathcal{N}(\calU, d, \varepsilon)^{\nicefrac{1}{2}} \, d\varepsilon\right).
\end{equation*}
\end{ethm}

As we show in the next proposition, we may translate this result into an upper bound on the expectation of the adapted optimal transport costs between $\hat{\mu}_{k,n}$ and $\mu_k$, and $\hat{\nu}_{k,n}$ and $\nu_k$ under a $\phi$-mixing assumption.

\begin{prop}\label{prop:mean_ot_bound}
Let $\mu$ and $\nu$ have $\phi$-mixing coefficients $\phi_\mu$ and $\phi_\nu$, respectively.
Then, there exists a constant $C < \infty$ such that for any $n \geq 1$, $k \in \{1, ..., n\}$, and $t \in (0, \frac{1}{4}\|c\|_\infty]$,
\begin{align*}
&\mathbb{E}\left[\frac{1}{k} \otcXk(\hat{\mu}_{k,n}, \mu_k) + \frac{1}{k} \otcYk(\hat{\nu}_{k,n}, \nu_k)\right] \\
&\quad\quad\leq C \left(t + \left(\frac{1}{n^2} \sum\limits_{g=0}^n (n-g+1) \phi_\mu^{\nicefrac{1}{2}}(g)\right)^{\nicefrac{1}{2}} \int_t^{\frac{1}{4} \|c\|_\infty} \mathcal{N}\left(\X^k, \frac{1}{k} \cXk, \varepsilon\right)^{\nicefrac{1}{2}} \, d\varepsilon \right. \\
&\quad\quad\quad \left.+ \left(\frac{1}{n^2} \sum\limits_{\ell=0}^n (n-g+1) \phi_\nu^{\nicefrac{1}{2}}(g)\right)^{\nicefrac{1}{2}} \int_t^{\frac{1}{4} \|c\|_\infty} \mathcal{N}\left(\Y^k, \frac{1}{k} \cYk, \varepsilon\right)^{\nicefrac{1}{2}} \, d\varepsilon \right).
\end{align*}
\end{prop}
\begin{proof}
The result follows from two applications of Theorem \ref{thm:rho_mixing_bound} for $\mu$ and $\nu$.
Considering first the case of $\mu$, let $\calU = \X^k$ with pseudo-metric $d = \frac{1}{k} \cXk$.
Moreover, let $\tilde{\mu} \in \calM_s((\X^k)^\bbN)$ be the distribution of the stationary process $(X_1, ..., X_k), (X_2, ..., X_{k+1}), ...$.
Note that the $\rho$- and $\phi$-mixing coefficients of $\tilde{\mu}$ satisfy $\rho_{\tilde{\mu}}(g) \leq 2 \phi_{\tilde{\mu}}^{1/2}(g)$ for every $g \geq 0$ \citep{bradley2005basic}.
One may also easily establish that $\phi_{\tilde{\mu}}(g) \leq \phi_{\mu}(g)$ for every $g \geq 0$.
Then a direct application of Theorem \ref{thm:rho_mixing_bound} to $\tilde{\mu}$ yields
\begin{align*}
\mathbb{E}\left[\frac{1}{k} \otcXk(\hat{\mu}_{k,n}, \mu_k)\right] &\leq C \left( t + \left(\frac{2}{n^2} \sum\limits_{m=0}^n \sum\limits_{g=0}^m \phi_\mu^{\nicefrac{1}{2}}(g)\right)^{\nicefrac{1}{2}} \int_t^{\frac{1}{4} \Delta} \mathcal{N}\left(\X^k, \frac{1}{k} \cXk, \varepsilon\right)^{\nicefrac{1}{2}} \, d\varepsilon\right) \\
&= C \left( t + \left(\frac{2}{n^2} \sum\limits_{g=0}^n (n-g+1) \phi_\mu^{\nicefrac{1}{2}}(g)\right)^{\nicefrac{1}{2}} \int_t^{\frac{1}{4} \Delta} \mathcal{N}\left(\X^k, \frac{1}{k} \cXk, \varepsilon\right)^{\nicefrac{1}{2}} \, d\varepsilon\right), 
\end{align*}

\noindent for some constant $C < \infty$ and any $t \in (0, \frac{1}{4}\Delta]$, where the $n-k+1$ term comes from the fact that there are $n-k+1$ $k$-blocks in the sequence $X_1^n$.
In this case $\Delta = \|c\|_\infty$ and an identical argument for $\nu$ yields the result.
\end{proof}

\paragraph{Proof of Main Results.}
Gathering the results proven above, we may proceed with the proofs of Theorem \ref{thm:abstract_rates} and Corollary \ref{cor:rates}.

\abstractrates*
\begin{proof}
Let $n$, $k$ and $g$ be as in the statement of the theorem.
By the triangle inequality, 
\begin{equation*}
\left|\ojhatk(X_1^n, Y_1^n) - \oj(\mu, \nu)\right| \leq \left|\ojhatk(X_1^n, Y_1^n) - \frac{1}{k} \otck(\mu_k, \nu_k)\right| + \left|\frac{1}{k} \otck(\mu_k, \nu_k) - \oj(\mu, \nu)\right|.
\end{equation*}
We begin by establishing an upper bound on $\left|\frac{1}{k} \otck(\mu_k, \nu_k) - \oj(\mu, \nu)\right|$.
Note by Proposition \ref{prop:oj_is_otcbar} that $\frac{1}{k} \otck(\mu_k, \nu_k) \leq \oj(\mu, \nu)$, so it suffices to upper bound $\oj(\mu, \nu) - \frac{1}{k} \otck(\mu_k, \nu_k)$.
Let $\pi \in \Pi(\mu_k, \nu_k)$ achieve the minimum in the problem $\otck(\mu_k, \nu_k)$ and let $\gamma \in \calM(\X^g \times \Y^g)$ be a probability measure with marginals $\alpha \in \calM(\X^g)$ and $\beta \in \calM(\Y^g)$.
Finally, let $\lambda^{k,g} \in \calM_s(\X^\bbN \times \Y^\bbN)$ be the stationary process measure satisfying $\lambda^{k,g} = \Lambda^{k+g}[\pi \otimes \gamma]$.
In other words, $\lambda^{k,g}$ is obtained by independently concatenating $\pi$ and $\gamma$ infinitely many times and randomizing the start over the first $k+g$ coordinates.
Note that $\lambda^{k,g} \in \J(\Lambda^k[\mu, \alpha], \Lambda^k[\nu, \beta])$.
Then by the construction of $\lambda^{k,g}$,
\begin{align*}
\oj\left(\Lambda^k[\mu, \alpha], \Lambda^k[\nu, \beta]\right) \leq \int c \, d\lambda^{k,g} = \frac{1}{k+g} \left(\int c_k \, d\pi + \int c_g \, d\gamma\right) \leq \frac{1}{k+g} \left(\otck(\mu_k, \nu_k) + g \|c\|_\infty \right).
\end{align*}
Rearranging terms, multiplying by $\nicefrac{(k+g)}{k}$, and adding $\oj(\mu, \nu)$ to both sides, we obtain that
\begin{align}
\oj(\mu, \nu) -\frac{1}{k}\otck(\mu_k, \nu_k) &\leq \oj(\mu, \nu) - \frac{k+g}{k} \oj\left(\Lambda^k[\mu, \alpha], \Lambda^k[\nu, \beta]\right) + \frac{g}{k} \|c\|_\infty \notag \\
& \leq \oj(\mu, \nu) - \oj\left(\Lambda^k[\mu, \alpha], \Lambda^k[\nu, \beta]\right) + \frac{g}{k} \|c\|_\infty, \label{eq:nonneg_oj}
\end{align}
where in \eqref{eq:nonneg_oj} we use the fact that the cost $c$, and thus the optimal joining cost, is non-negative.
By Lemmas \ref{lemma:oj_is_lipschitz} and \ref{lemma:d_bar_rates}, we see that
\begin{align*}
\oj(\mu, \nu) - \oj\left(\Lambda^k[\mu, \alpha], \Lambda^k[\nu, \beta]\right) & \leq \|c\|_\infty \left(\overline{d}\left(\Lambda^k[\mu, \alpha], \mu\right) + \overline{d}\left(\Lambda^k[\nu, \beta], \nu\right)\right) \\
&\leq  \|c\|_\infty \left(\frac{2g}{k+g} + \frac{k}{k+g}(\phi_\mu(g+1) + \phi_\nu(g+1))\right) \\
&\leq  \|c\|_\infty \left(\frac{2g}{k} + \frac{k}{k+g}(\phi_\mu(g+1) + \phi_\nu(g+1))\right).
\end{align*}
It follows that
\begin{equation*}
\left|\frac{1}{k} \otck(\mu_k, \nu_k) - \oj(\mu, \nu)\right| \leq \|c\|_\infty \left(\frac{k (\phi_\mu(g+1) + \phi_\nu(g+1))}{k+g} + \frac{3g}{k}\right).
\end{equation*}
In order to bound the other term, we apply Proposition \ref{prop:eot_is_lipschitz} to obtain the bound
\begin{equation*}
\left|\ojhatk(X_1^n, Y_1^n) - \frac{1}{k} \otck(\mu_k, \nu_k)\right| \leq \frac{1}{k} \otcXk(\hat{\mu}_{k,n}, \mu_k) + \frac{1}{k} \otcYk(\hat{\nu}_{k,n}, \nu_k).
\end{equation*}
Combining the bounds proven above, taking an expectation, and applying Proposition \ref{prop:mean_ot_bound}, we obtain the result.
\end{proof}

\rates*
\begin{proof}
The summability condition for $\phi_\mu$ implies that $n^{-2} \sum_{\ell=0}^n (n-\ell+1) \phi_\mu^{\nicefrac{1}{2}}(\ell) = \mathcal{O}(n^{p-2})$ and thus $(n^{-2} \sum_{\ell=0}^n (n-\ell+1) \phi_\mu^{\nicefrac{1}{2}}(\ell))^{\nicefrac{1}{2}} = \mathcal{O}(n^{\nicefrac{p}{2}-1})$.
Moreover, for every $\varepsilon \in (0, \frac{1}{4} \|c\|_\infty]$, we have $\mathcal{N}(\X^k, \frac{1}{k} \cXk, \varepsilon) \leq |\X|^k$.
Thus for large enough $n$, there is a constant $C < \infty$ such that 
\begin{equation*}
u_t(k, n) \leq C n^{\nicefrac{p}{2}-1} \left(\frac{1}{4} \|c\|_\infty - t\right) |\X|^{\nicefrac{k}{2}} \leq \frac{C |\X|^{\nicefrac{k}{2}}}{n^{1-\nicefrac{p}{2}}}.
\end{equation*}
Using the same line of reasoning, one may prove the analogous bound for $v_t(k, n)$.
Plugging these bounds into Theorem \ref{thm:abstract_rates}, we obtain the result by letting $t \rightarrow 0$.
\end{proof}

\subsection{Proofs from Section \ref{sec:entropic_oj}}
\label{sec:proofs_eoj}
In this section, we prove the results stated in Section \ref{sec:entropic_oj}.
We begin with Lemma \ref{lemma:subadditivity_of_eot}, which states that the limit in Proposition \ref{prop:convergence_of_eot_to_eoj} exists and is equal to a supremum.

\begin{lem}\label{lemma:subadditivity_of_eot}
For any $\eta \geq 0$,
\begin{equation*}
\lim\limits_{k\rightarrow\infty} \frac{1}{k} \eotck(\mu_k, \nu_k) = \sup\limits_{k\geq 1} \frac{1}{k} \eotck(\mu_k, \nu_k).
\end{equation*}

\end{lem}
\begin{proof}
By Fekete's lemma, it suffices to show that the sequence $\{\eotck(\mu_k, \nu_k)\}_{k \geq 1}$ is superadditive.
Fix $k, \ell \geq 1$ and let $\pi \in \Pi(\mu_{k+\ell}, \nu_{k+\ell})$ be a solution to $\eot_{c_{k+\ell}}(\mu_{k+\ell}, \nu_{k+\ell})$.
Let $\pi_k \in \calM(\X^k \times \Y^k)$ and $\pi_\ell \in \calM(\X^\ell \times \Y^\ell)$ be the measures corresponding to the first $k$ coordinates and last $\ell$ coordinates of $\pi$, respectively.
Using the stationarity of $\mu$ and $\nu$, it is straightforward to show that $\pi_k \in \Pi(\mu_k, \nu_k)$ and $\pi_\ell \in \Pi(\mu_\ell, \nu_\ell)$.
Moreover, using the subadditivity of $H_{k+\ell}(\cdot)$, 
\begin{align*}
\eot_{c_{k+\ell}}(\mu_{k+\ell}, \nu_{k+\ell}) &= \int c_{k+\ell} \, d\pi - \eta H_{k+\ell}(\pi) \\
&\geq \int c_k \, d\pi_k - \eta H_k(\pi_k) + \int c_\ell \, d\pi_\ell - \eta H_\ell(\pi_\ell) \\
&\geq \eotck(\mu_k, \nu_k) + \eot_{c_\ell}(\mu_\ell, \nu_\ell).
\end{align*}

\noindent So the sequence $\{\eotck(\mu_k, \nu_k)\}_{k \geq 1}$ is superadditive and the conclusion follows.
\end{proof}

\convergenceofeottoeoj*
\begin{proof}
Fix $\varepsilon > 0$ and $\eta \geq 0$ and let $\lambda \in \mathcal{J}(\mu, \nu)$ be a joining of $\mu$ and $\nu$ such that
\begin{equation*}
\int c \, d\lambda_1 - \eta h(\lambda) \leq \eoj(\mu, \nu) + \varepsilon.
\end{equation*}

\noindent Since the $k$-dimensional distribution of $\lambda$, written as $\lambda_k$, satisfies $\lambda_k \in \Pi(\mu_k, \nu_k)$, we have
\begin{equation*}
\eotck(\mu_k, \nu_k) \leq  \int c_k \, d\lambda_k - \eta H(\lambda_k).
\end{equation*}

\noindent As $\lambda$ is stationary and $h(\lambda) \leq \frac{1}{k} H(\lambda_k)$,
\begin{align*}
\eoj(\mu, \nu) + \varepsilon &\geq \int c \, d\lambda_1 - \eta h(\lambda) \\
&= \frac{1}{k} \int c_k \, d\lambda_k - \eta h(\lambda) \\
& \geq \frac{1}{k} \int c_k \, d\lambda_k - \frac{\eta}{k} H(\lambda_k) \\
&\geq \frac{1}{k} \eotck(\mu_k, \nu_k).
\end{align*}

\noindent By Lemma \ref{lemma:subadditivity_of_eot} we may take a limit in $k$ and let $\varepsilon \rightarrow 0$ to establish that
\begin{equation*}
\eoj(\mu, \nu) \geq \lim\limits_{k\rightarrow\infty} \frac{1}{k} \eotck(\mu_k, \nu_k).
\end{equation*}

\noindent Now let $\{\pi^k\}$ be a sequence with $\pi^k \in \Pi(\mu_k, \nu_k)$ such that
\begin{equation*}
\frac{1}{k}\int c_k \, d\pi^k - \frac{\eta}{k} H(\pi_k) \leq \frac{1}{k}\eotck(\mu_k, \nu_k) + \varepsilon_k,
\end{equation*}

\noindent where $\varepsilon_k \rightarrow 0$.
From this sequence, we wish to construct a sequence of joinings converging to a joining of $\mu$ and $\nu$.
For every $k \geq 1$, let $\lambda^k \in \calM_s(\X^\bbN \times \Y^\bbN)$ be the stationary process measure satisfying $\lambda^k = \Lambda^k[\pi^k]$.
We will now show that the $\X^\bbN$- and $\Y^\bbN$-marginals of $\lambda^k$ converge weakly to $\mu$ and $\nu$.
Let $\sigma: \X^\bbN \rightarrow \X^\bbN$ and $\tau: \Y^\bbN \rightarrow \Y^\bbN$ be the left-shift maps on $\X^\bbN$ and $\Y^\bbN$, respectively.
Fix a measurable cylinder set $C = C_1 \times \cdots \times C_m \subset \X^m$ and let $\tilde{C} \subset \XN$ be its extension to $\XN$ such that $\tilde{C} = C \times \X \times \X \cdots$.
Then for $k \geq m$,
\begin{align*}
\lambda^k(\tilde{C} \times \YN) &= \frac{1}{k} \sum\limits_{\ell=0}^{k-1} \tilde{\Lambda}^k[\pi^k](\sigma^{-\ell}\tilde{C} \times \tau^{-\ell}\YN) \\
&= \frac{1}{k}\sum\limits_{\ell=0}^{k-1}\tilde{\Lambda}^k[\pi^k](\sigma^{-\ell}\tilde{C}\times \YN) \\
&= \frac{1}{k}\sum\limits_{\ell=0}^{k-1} \tilde{\Lambda}^k[\mu_k](\sigma^{-\ell}\tilde{C}) \\
&= \frac{k-m+1}{k} \mu_m(C) + \frac{1}{k} \sum\limits_{\ell=1}^{m-1} \mu_{m-\ell}\left(C_1 \times \cdots \times C_{m-\ell}\right)\mu_\ell\left(C_{m-\ell+1} \times \cdots \times C_m\right).
\end{align*}

\noindent Fixing $m$ and taking a limit in $k$, we see that 
\begin{equation*}
\lim\limits_{k\rightarrow\infty} \lambda^k(\tilde{C}\times \YN) = \mu_m(C) = \mu(\tilde{C}).
\end{equation*}

\noindent Thus the $\X^\bbN$-marginal of $\lambda^k$ converges weakly to $\mu$ and one may use a similar argument to show that the $\Y^\bbN$-marginal of $\lambda^k$ converges weakly to $\nu$.
So by Lemma \ref{lemma:weakly_convergent_subsequence}, $\lambda^{k_\ell} \Rightarrow \lambda \in \mathcal{J}(\mu, \nu)$ for some subsequence $\{\lambda^{k_\ell}\}$.
Now for each $\ell \geq 1$, one may show using the definition of $\lambda^{k_\ell}$ that $\int c \, d\lambda^{k_\ell}_1 = \frac{1}{k_\ell} \int c_{k_\ell} \, d\pi^{k_\ell}$ and $h(\lambda^{k_\ell}) = \frac{1}{k_\ell} H(\pi^{k_\ell})$.
Thus, by the upper semicontinuity of $h(\cdot)$ and the continuity and boundedness of $c$,
\begin{align*}
\eoj(\mu, \nu) &\leq \int c \, d\lambda_1 - \eta h(\lambda) \\
&\leq \liminf\limits_{\ell\rightarrow \infty} \left\{\int c \, d\lambda^{k_\ell}_1  - \eta h(\lambda^{k_\ell})\right\} \\
&= \liminf\limits_{\ell\rightarrow \infty} \left\{\frac{1}{k_\ell} \int c_{k_\ell} \, d\pi^{k_\ell}  - \frac{\eta}{k_\ell} H(\pi^{k_\ell})\right\} \\
&\leq \liminf\limits_{\ell\rightarrow\infty}\left\{\frac{1}{k_\ell} \ot^\eta_{c_{k_\ell}}(\mu_{k_\ell}, \nu_{k_\ell}) + \varepsilon_{k_\ell}\right\} \\
&= \lim\limits_{k\rightarrow\infty} \frac{1}{k} \eotck(\mu_k, \nu_k),
\end{align*}

\noindent giving the result.
\end{proof}

\convergenceineta*
\begin{proof}
Let $\{\eta_n\}$ be a sequence of non-negative integers such that $\eta_n \rightarrow 0$ and for every $n \geq 1$ let $\lambda^n \in \Jmin^{\eta_n}(\mu, \nu)$.
As $\J(\mu, \nu)$ is compact in the weak topology, there exists a subsequence of $\{\lambda^n\}$, which we also refer to as $\{\lambda^n\}$, such that $\lambda^n \Rightarrow \lambda$ for some $\lambda \in \J(\mu, \nu)$.
Now let $\lambda^* \in \Jmin(\mu, \nu)$.
Using the feasibility of $\lambda^n$ for $\oj(\mu, \nu)$ and $\lambda^*$ for $\oj^{\eta_n}(\mu, \nu)$, it follows that for every $n \geq 1$,
\begin{equation*}
\int c \, d\lambda^*_1 \leq \int c \, d\lambda^n_1
\end{equation*}

\noindent and
\begin{equation*}
\int c \, d\lambda^n_1 - \eta_n h(\lambda^n) \leq \int c \, d\lambda^*_1 - \eta_n h(\lambda^*).
\end{equation*}

\noindent Rearranging, we obtain
\begin{equation}\label{eq:sandwich_ineq}
0 \leq \int c \, d\lambda^n_1 - \int c \, d\lambda^*_1 \leq \eta_n (h(\lambda^n) - h(\lambda^*)).
\end{equation}

\noindent As $h(\cdot)$ is bounded, we have $\lim_{n\rightarrow \infty} \eta_n (h(\lambda^n) - h(\lambda^*)) = 0$.
Taking limits in \eqref{eq:sandwich_ineq} and using the continuity and boundedness of $c$,
\begin{equation*}
\int c \, d\lambda_1 = \lim\limits_{n\rightarrow\infty} \int c \, d\lambda^n_1 = \int c \, d\lambda^*_1 = \oj(\mu, \nu).
\end{equation*}

\noindent It follows that $\lambda \in \Jmin(\mu, \nu)$ and $\lim_{n\rightarrow\infty} \int c\, d\lambda_1^n = \oj(\mu, \nu)$.
Again using the boundedness of $h(\cdot)$,
\begin{equation*}
\lim\limits_{n\rightarrow\infty} \oj^{\eta_n}(\mu, \nu) = \lim\limits_{n\rightarrow\infty} \left\{\int c \, d\lambda^n_1  - \eta_n h(\lambda^n)\right\} = \lim\limits_{n\rightarrow\infty} \int c \, d\lambda^n_1 = \oj(\mu, \nu).
\end{equation*}

\noindent Since $\{\eta_n\}$ was arbitrary, we obtain the result.
\end{proof}

\consistentestimationofeoj*
\begin{proof}
Fix some $\eta > 0$.
To begin, we would like to construct a sequence $\{k(n)\}$ satisfying
\begin{equation}\label{eq:consistency_of_cost}
\lim\limits_{n\rightarrow\infty} \left|\eojhatk(X_1^n, Y_1^n) - \eoj\left(\mu, \nu\right)\right| = 0, \quad \mbox{almost surely}.
\end{equation}
\noindent Our approach will be similar to the unregularized case with the exception that we also have to control the error in the entropies of the estimates.
In particular, by Proposition \ref{prop:eot_is_lipschitz},
\begin{align*}
\left| \eojhatk(X_1^n, Y_1^n) - \frac{1}{k} \eotck(\mu_k, \nu_k)\right| &\leq \frac{1}{k} \otcXk(\hat{\mu}_{k,n}, \mu_k) + \eta \left|\frac{1}{k}H(\hat{\mu}_{k,n}) - \frac{1}{k} H(\mu_k)\right| \\
&\quad\quad + \frac{1}{k} \otcYk(\hat{\nu}_{k,n}, \nu_k) + \eta\left|\frac{1}{k} H(\hat{\nu}_{k,n}) - \frac{1}{k} H(\nu_k)\right|.
\end{align*}
\noindent So it is necessary to ensure that the entropy error terms also decay to zero along the sequence $\{k(n)\}$ that we construct.
To see that such a sequence exists, fix $\varepsilon > 0$ and $i \in \mathbb{N}$ and recall that since $\hat{\mu}_{i,n} \Rightarrow \mu_i$ and $H(\cdot)$ is weakly continuous, there exists an $n(i) \in \mathbb{N}$ such that
\begin{equation*}
\mu\left(\left|\frac{1}{i} H(\hat{\mu}_{i,n(i)}) - \frac{1}{i} H(\mu_i)\right| > \varepsilon\right) \leq 2^{-i}.
\end{equation*}
\noindent Then by Borel-Cantelli,
\begin{equation*}
\mu\left(\limsup\limits_{i\rightarrow\infty} \left|\frac{1}{i} H(\hat{\mu}_{i,n(i)}) - \frac{1}{i} H(\mu_i)\right| > \varepsilon \right) = 0,
\end{equation*}
and it follows that $\left|\frac{1}{i} H(\hat{\mu}_{i,n(i)}) - \frac{1}{i} H(\mu_i)\right| \rightarrow 0$, $\mu$-almost surely.
Abusing notation somewhat, we obtain a sequence $\{i(n)\}$ by letting $i(n') = \inf_i\{n(i) = n'\}$.
Thus 
\begin{equation*}
\lim_{n\rightarrow\infty}\left|\frac{1}{i(n)} H(\hat{\mu}_{i(n),n}) - \frac{1}{i(n)} H(\mu_{i(n)})\right| \rightarrow 0, \quad\quad \mu\mbox{-almost surely.}
\end{equation*}
Let $\{j(n)\}$ be a sequence constructed in the analogous manner for $\nu$ and let $\{\ell(n)\}$ and $\{m(n)\}$ be admissible sequences for $\mu$ and $\nu$.
Then letting $\{k(n)\}$ be the sequence defined by $k(n) = \min\{$ $i(n),$ $j(n),$ $\ell(n),$ $m(n)\}$, we obtain the desired convergence.

Next we show that the sequence of estimated entropic optimal joinings indexed by $k(n)$ converges weakly to the set of entropic optimal joinings $\Jmin^\eta(\mu, \nu)$, almost surely.
In order to simplify notation, we will suppress the dependence of $k(n)$ on $n$ in the rest of the proof.
Fix sequences $X_1, X_2, ...$ and $Y_1, Y_2, ...$ in the sets of $\mu$- and $\nu$-measure one on which \eqref{eq:consistency_of_cost} holds.
Let $\{\esteojkn\}_{n\geq 1}$ be the corresponding sequence of estimated entropic optimal joinings. 
By Lemma \ref{lemma:weakly_convergent_subsequence} in Appendix \ref{app:weakconvergence}, for any subsequence $\{\esteoj^{k, n_\ell}\}_{\ell \geq 1}$, there is a further subsequence converging weakly to a joining $\lambda \in \J(\mu, \nu)$.
For ease of notation, we refer to this further subsequence again as $\{\esteoj^{k, n_\ell}\}_{\ell \geq 1}$.
Using the upper semicontinuity of $h(\cdot)$ and the continuity and boundedness of $c$, one may establish
\begin{align*}
\int c \, d\lambda_1 - \eta h(\lambda) &\leq \liminf\limits_{\ell \rightarrow\infty} \left\{\int c \, d\esteoj^{k, n_\ell}_1 - \eta h(\esteoj^{k, n_\ell})\right\} \\
&= \liminf\limits_{\ell \rightarrow\infty} \frac{1}{k} \eotck(\mu^{n_\ell}_{k}, \nu^{n_\ell}_{k}) \\
&= \eot(\mu, \nu).
\end{align*}

\noindent Thus, $\lambda \in \Jmin^\eta(\mu, \nu)$ and since the subsequence was arbitrary, we conclude that $\esteojkn \Rightarrow \Jmin^\eta(\mu, \nu)$.
By the choice of sequences $X_1, X_2, ...$ and $Y_1, Y_2, ...$, we conclude that $\esteojkn \Rightarrow \Jmin^\eta(\mu, \nu)$ almost surely.
\end{proof}

Before proving Theorem \ref{thm:eoj_error_bound}, we state a lemma regarding the entropy of the $k$-block empirical measure $\hat{\mu}_{k,n}$.

\begin{restatable}[]{lem}{biasofentropy}
\label{lem:bias_of_entropy}
Suppose that for some $u: \mathbb{N} \times \mathbb{N} \rightarrow \mathbb{R}_+$ with $\lim_{n\rightarrow\infty} u(k,n) = 0$ for every $k \geq 1$, it holds that
\begin{equation*}
\mathbb{E}\|\mu_k - \hat{\mu}_{k,n}\|_1 \leq u(k,n)
\end{equation*}
for large enough $n$.
Then for any $k \geq 1$ and $n$ large enough,
\begin{equation*}
\mathbb{E}|H(\mu_k) - H(\hat{\mu}_{k,n})| \leq u(k,n) \log \left(\frac{|\mathcal{X}|^{3k}}{u(k,n)}\right).
\end{equation*}
\end{restatable}

A proof of Lemma \ref{lem:bias_of_entropy} may be found in Appendix \ref{app:entropy}.

\eojerrorbound*
\begin{proof}
Let $\eta \geq 0$ and fix $k \geq 1$, $g \geq 0$, and $n \geq k$.
By the triangle inequality,
\begin{equation*}
\left|\eojhatk(X_1^n, Y_1^n) - \eoj(\mu, \nu)\right| \leq \underbrace{\left|\eojhatk(X_1^n, Y_1^n) - \frac{1}{k} \eotck(\mu_k, \nu_k)\right|}_{\scriptsize T_1} + \underbrace{\left|\frac{1}{k} \eotck(\mu_k, \nu_k) - \eoj(\mu, \nu)\right|}_{\scriptsize T_2},
\end{equation*}
where $T_1$ and $T_2$ denote the two differences on the right hand side, respectively.
By Proposition \ref{prop:eot_is_lipschitz}, $T_1$ is bounded as
\begin{align*}
T_1 &\leq \frac{1}{k} \otcXk(\hat{\mu}_{k,n}, \mu_k) + \frac{\eta}{k} |H(\mu_k) - H(\hat{\mu}_{k,n})| + \frac{1}{k} \otcYk(\hat{\nu}_{k,n}, \nu_k) + \frac{\eta}{k} |H(\nu_k) - H(\hat{\nu}_{k,n})| \\
&\leq \frac{\|c\|_\infty}{2} (\|\hat{\mu}_{k,n} - \mu_k\|_1 + \|\hat{\nu}_{k,n} - \nu_k\|_1) + \frac{\eta}{k} (|H(\mu_k) - H(\hat{\mu}_{k,n})| + |H(\nu_k) - H(\hat{\nu}_{k,n})|).
\end{align*}
Taking an expectation of both sides, we obtain
\begin{equation*}
\mathbb{E} T_1 \leq \frac{\|c\|_\infty}{2} (\mathbb{E}\|\hat{\mu}_{k,n} - \mu_k\|_1 + \mathbb{E}\|\hat{\nu}_{k,n} - \nu_k\|_1) + \frac{\eta}{k} (\mathbb{E}|H(\mu_k) - H(\hat{\mu}_{k,n})| + \mathbb{E}|H(\nu_k) - H(\hat{\nu}_{k,n})|).
\end{equation*}
As the $\ell_1$ norm is equivalent (up to a constant factor) to the optimal transport cost with respect to the Hamming cost $\delta$, one has $\mathbb{E}\|\hat{\mu}_{k,n} - \mu_k\|_1 = 2\mathbb{E} \ot_\delta(\hat{\mu}_{k,n}, \mu_k)$.
Then one may repeat the arguments in the proof of Proposition \ref{prop:mean_ot_bound} with $\frac{1}{k} \cXk$ replaced by $\delta$ to show that there exists a constant $C < \infty$ such that
\begin{equation*}
\mathbb{E}\|\hat{\mu}_{k,n} - \mu_k\|_1 \leq C \left(t + \left(\frac{1}{n^2} \sum\limits_{\ell=0}^n (n-\ell+1)\phi^{\nicefrac{1}{2}}_\mu(\ell)\right)^{\nicefrac{1}{2}} \int_t^{\nicefrac{1}{4}} \mathcal{N}(\X^k, \delta, \varepsilon)^{\nicefrac{1}{2}} \, d\varepsilon \right),
\end{equation*}
for any $t \in (0, \nicefrac{1}{4}]$.
Using the summability condition for $\phi_\mu$ and the fact that $\mathcal{N}(\X^k, \delta, \varepsilon) \leq |\X|^k$ for every $\varepsilon > 0$, we obtain for large enough $n$,
\begin{equation*}
\mathbb{E}\|\hat{\mu}_{k,n} - \mu_k\|_1 \leq C \left(t + \left(\frac{1}{4} - t\right)\frac{|\X|^{\nicefrac{k}{2}}}{n^{1-\nicefrac{p}{2}}}\right),
\end{equation*}
where we have absorbed any additional constants arising from the summability condition on $\phi_\mu$ into $C$.
Finally, taking the limit $t \rightarrow 0$, we have that for large enough $n$, $\mathbb{E}\|\hat{\mu}_{k,n} - \mu_k\|_1 \leq u(k, n)$ where $u(k, n) = C |\X|^{\nicefrac{k}{2}} n^{\nicefrac{p}{2}-1}$, thereby satisfying the conditions of Lemma \ref{lem:bias_of_entropy}.
Using the same arguments, one may also establish the analogous bound $\mathbb{E}\|\hat{\nu}_{k,n} - \nu_k\|_1 \leq v(k,n)$ where $v(k, n) = C |\Y|^{\nicefrac{k}{2}} n^{\nicefrac{p}{2}-1}$.
Thus, Lemma \ref{lem:bias_of_entropy} implies
\begin{equation*}
\mathbb{E} T_1 \leq u(k, n) \left(\frac{\|c\|_\infty}{2} + \frac{\eta}{k} \log \left(\frac{|\mathcal{X}|^{3k}}{u(k, n)}\right)\right) + v(k, n) \left(\frac{\|c\|_\infty}{2} + \frac{\eta}{k} \log \left(\frac{|\mathcal{Y}|^{3k}}{v(k, n)}\right)\right),
\end{equation*}
for $n$ large enough.

Next we upper bound $T_2$.
Note that according to Lemma \ref{lemma:subadditivity_of_eot} and Proposition \ref{prop:convergence_of_eot_to_eoj}, $k^{-1} \eotck(\mu, \nu) \leq \eoj(\mu, \nu)$.
Thus it suffices to bound $\eoj(\mu, \nu) - k^{-1} \eotck(\mu, \nu)$.
In order to do this, we will approximate an optimal joining of $\mu$ and $\nu$ by an iid block process with gaps.
Let $\pi \in \Pi(\mu_k, \nu_k)$ achieve the minimum in the problem $\eotck(\mu_k, \nu_k)$ and let $\gamma \in \calM(\X^g \times \Y^g)$ be a probability measure with marginals $\alpha \in \calM(\X^g)$ and $\beta \in \calM(\Y^g)$.
Finally, let $\lambda^{k,g} \in \calM_s(\X^\bbN \times \Y^\bbN)$ be the stationary process measure satisfying $\lambda^{k,g} = \Lambda^{k+g}[\pi \otimes \gamma]$.
In other words, $\lambda^{k,g}$ is the stationary process obtained by independently concatenating $\pi \otimes \gamma$ and randomizing the start over the first $k+g$ symbols.
Note that $\lambda^{k,g} \in \J(\Lambda^{k,g}[\mu, \alpha], \Lambda^{k,g}[\nu, \beta])$ and by Lemma \ref{lem:invariance_of_entropy_rate},
\begin{equation*}
h(\lambda^{k,g}) = h(\Lambda^k[\pi \otimes \gamma]) = h(\tilde{\Lambda}^k[\pi \otimes \gamma]) = \frac{1}{k+g} (H(\pi) + H(\gamma)).
\end{equation*}
Then by the construction of $\lambda^{k,g}$,
\begin{align*}
\eoj\left(\Lambda^k[\mu, \alpha], \Lambda^k[\nu, \beta]\right) &\leq \int c \, d\lambda^{k,g} - \eta h(\lambda^{k,g})\\ 
&= \frac{1}{k+g} \left(\int c_k \, d\pi + \int c_g \, d\gamma - \eta H(\pi) - \eta H(\gamma)\right) \\
&\leq \frac{1}{k+g} \left(\eotck(\mu_k, \nu_k) + g \|c\|_\infty \right).
\end{align*}
Rearranging terms, multiplying by $\nicefrac{(k+g)}{k}$, and adding $\eoj(\mu, \nu)$ to both sides, we obtain
\begin{align*}
T_2 &\leq \eoj(\mu, \nu) - \frac{k+g}{k} \eoj\left(\Lambda^k[\mu, \alpha], \Lambda^k[\nu, \beta]\right) + \frac{g}{k} \|c\|_\infty \\
&= \eoj(\mu, \nu) - \eoj\left(\Lambda^k[\mu, \alpha], \Lambda^k[\nu, \beta]\right) - \frac{g}{k} \eoj\left(\Lambda^k[\mu, \alpha], \Lambda^k[\nu, \beta]\right) + \frac{g}{k} \|c\|_\infty.
\end{align*}
By Lemmas \ref{lemma:oj_is_lipschitz} and \ref{lemma:d_bar_rates}, we see that
\begin{align*}
&\eoj(\mu, \nu) - \eoj\left(\Lambda^k[\mu, \alpha], \Lambda^k[\nu, \beta]\right) \\
&\quad \leq \|c\|_\infty \left(\overline{d}\left(\Lambda^k[\mu, \alpha], \mu\right) + \overline{d}\left(\Lambda^k[\nu, \beta], \nu\right)\right) + \eta (h(\mu) - h(\Lambda^k[\mu, \alpha]) + h(\nu) - h(\Lambda^k[\nu, \beta])) \\
&\quad \leq  \|c\|_\infty \left(\frac{2g}{k+g} + \frac{k}{k+g}(\phi_\mu(g+1) + \phi_\nu(g+1))\right) + \eta (h(\mu) - h(\Lambda^k[\mu, \alpha]) + h(\nu) - h(\Lambda^k[\nu, \beta])).
\end{align*}
In order to simplify this bound, note that
\begin{align*}
h(\mu) - h(\Lambda^k[\mu, \alpha]) &= h(\mu) - \frac{1}{k+g} (H(\mu_k) + H(\alpha)) \\
&\leq \frac{1}{k} H(\mu_k) - \frac{1}{k+g} H(\mu_k) \\
&= \frac{g}{k(k+g)} H(\mu_k) \\
&\leq \frac{g \log |\X|^k}{k(k+g)} \\
&= \frac{g \log |\X|}{k+g},
\end{align*}
and similarly $h(\nu) - h(\Lambda^k[\nu, \beta]) \leq \frac{g \log |\Y|}{k+g}$.
Thus,
\begin{align*}
&\eoj(\mu, \nu) - \eoj\left(\Lambda^k[\mu, \alpha], \Lambda^k[\nu, \beta]\right) \\
&\quad \leq  \|c\|_\infty (\phi_\mu(g+1) + \phi_\nu(g+1))\frac{k}{k+g} + (2\|c\|_\infty + \eta(\log |\X| + \log |\Y|)) \frac{g}{k+g}.
\end{align*}
Moreover, using the fact that the independent joining has maximal entropy rate among all joinings with the same marginals, one may establish
\begin{equation*}
-\eoj\left(\Lambda^k[\mu, \alpha], \Lambda^k[\nu, \beta]\right) \leq \eta h(\Lambda^k[\mu, \alpha] \otimes \Lambda^k[\nu, \beta]) = \frac{\eta}{k+g} \left(H(\mu_k) + H(\alpha) + H(\nu_k) + H(\beta)\right).
\end{equation*}
As the measure $\gamma \in \calM(\X^g \times \Y^g)$ and marginals $\alpha \in \calM(\X^g)$ and $\beta \in \calM(\Y^g)$ were arbitrary, there is no loss of generality in assuming that $H(\alpha) = H(\beta) = 0$.
Thus
\begin{align*}
-\eoj\left(\Lambda^k[\mu, \alpha], \Lambda^k[\nu, \beta]\right) &\leq \frac{\eta}{k+g} \left(H(\mu_k) + H(\nu_k)\right) \\
&\leq \frac{\eta}{k+g} \left(\log |\X|^k + \log |\Y|^k\right) \\
&= \frac{\eta k}{k+g} \left(\log |\X| + \log |\Y|\right).
\end{align*}
Combining this and our earlier bounds, we have
\begin{align*}
T_2 &\leq \|c\|_\infty (\phi_\mu(g+1) + \phi_\nu(g+1))\frac{k}{k+g} + (2\|c\|_\infty + \eta(\log |\X| + \log |\Y|)) \frac{g}{k+g} \\
&\quad\quad + \frac{\eta g}{k} \left(\log |\X| + \log |\Y|\right) + \|c\|_\infty \frac{g}{k} \\
&\leq \|c\|_\infty (\phi_\mu(g+1) + \phi_\nu(g+1))\frac{k}{k+g} + (3\|c\|_\infty + 2\eta(\log |\X| + \log |\Y|)) \frac{g}{k}.
\end{align*}
Combining the bounds on $\mathbb{E}T_1$ and $T_2$ proven above, we obtain the result.
\end{proof}

\section*{Acknowledgments}
KO and ABN were supported in part by NSF grants DMS-1613072 and DMS-1613261.
KM gratefully acknowledges the support of NSF CAREER grant DMS-1847144.

\bibliographystyle{abbrv}
\bibliography{references}

\appendix

\section{Properties of Entropy and Entropy Rate}\label{app:entropy}
\begin{lem}\label{lem:invariance_of_entropy_rate}
Let $\calU$ be finite, let $\sigma$ be the left-shift on $\calU^\bbN$, and let $\tilde{\gamma} \in \calM(\calU^\bbN)$ be $n$-stationary, i.e. satisfying $\tilde{\gamma} \circ \sigma^{-n}= \tilde{\gamma}$.
Then if $\gamma = \frac{1}{n} \sum_{\ell=0}^{n-1} \tilde{\gamma} \circ \sigma^{-\ell}$, it holds that $h(\gamma) = h(\tilde{\gamma})$.
\end{lem}
\begin{proof}
The proof follows the argument outlined in \cite{shields1996ergodic}.
We find it most convenient to adopt the notation of random variables: let $\tilde{X} = \tilde{X}_1, \tilde{X}_2, ...$ be an $n$-stationary process with distribution $\tilde{\gamma}$, let $S$ be uniformly distributed on the set $\{0, ..., n-1\}$, and let $X = X_1, X_2, ...$ be the stationary process such that for any $s \in \{0, ..., n-1\}$, $X | S = s$ is equal in distribution to $\tilde{X}_s, \tilde{X}_{s+1}, ...$.
Note that, by construction, $X$ has distribution equal $\gamma$.
We will assume that $\tilde{X}$, $X$ and $S$ are all defined on a probability space $(\Omega, \mathcal{B}, \mathbb{P})$.

Fix some $k \geq 1$ and note that by the chain rule for entropy, we have
\begin{equation*}
H(X_1^k) - H(X_1^k | S) = H(S) - H(S | X_1^k).
\end{equation*}
This implies that
\begin{equation*}
\left|\frac{1}{k} H(X_1^k) - \frac{1}{k} H(X_1^k| S)\right| = \frac{1}{k} |H(S) - H(S | X_1^k)| \leq \frac{\log n}{k}.
\end{equation*}
Letting $k \rightarrow \infty$, we find
\begin{equation}\label{eq:invariance_of_entropy1}
h(X) = \lim\limits_{k\rightarrow\infty} \frac{1}{k} H(X_1^k) = \lim\limits_{k\rightarrow\infty} \frac{1}{k} H(X_1^k | S).
\end{equation}
Now writing out $H(X_1^k | S)$, 
\begin{equation}\label{eq:invariance_of_entropy2}
H(X_1^k | S) = \sum\limits_{s=0}^{n-1} \mathbb{P}(S = s) H(X_1^k | S = s) = \frac{1}{n} \sum\limits_{s=0}^{n-1} H(X_1^k | S = s) = \frac{1}{n} \sum\limits_{s=0}^{n-1} H(\tilde{X}_{s+1}^{s+k}).
\end{equation}
The rest of the argument will follow by showing that $\lim_{k\rightarrow\infty} \frac{1}{k} H(\tilde{X}_{s+1}^{s+k})$ is independent of $s$ and equal to the entropy rate $h(\tilde{X})$ of the $n$-stationary process $\tilde{X}$.
To see this, note that the chain rule for the entropy implies
\begin{equation*}
H(\tilde{X}_1^{s+k}) = H(\tilde{X}_{s+1}^{s+k}) + H(\tilde{X}_1^s | \tilde{X}_{s+1}^{s+k})
\end{equation*}
for every $s = 0, ..., n-1$.
This implies the bound 
\begin{equation*}
\left| \frac{1}{k} H(\tilde{X}_{s+1}^{s+k}) - \frac{1}{k} H(\tilde{X}_1^{s+k})\right| \leq \frac{\log |\mathcal{X}|^s}{k}.
\end{equation*}
Letting $k \rightarrow\infty$ once again, we have 
\begin{equation*}
\lim\limits_{k\rightarrow\infty} \frac{1}{k} H(\tilde{X}_{s+1}^{s+k}) = \lim\limits_{k\rightarrow\infty} \frac{1}{k} H(\tilde{X}_1^{s+k}) = \lim\limits_{k \rightarrow\infty} \frac{1}{s+k} H(\tilde{X}_1^{s+k}) = h(\tilde{X}).
\end{equation*}
Now, using this fact with \eqref{eq:invariance_of_entropy1} and \eqref{eq:invariance_of_entropy2}, we obtain
\begin{equation*}
h(X) = \frac{1}{n} \sum\limits_{s=0}^{n-1} \lim\limits_{k\rightarrow\infty} \frac{1}{k} H(\tilde{X}_{s+1}^{s+k}) = \frac{1}{n} \sum\limits_{s=0}^{n-1} h(\tilde{X}) = h(\tilde{X}).
\end{equation*}
Translating this result back to the measures $\gamma$ and $\tilde{\gamma}$, we obtain the desired result.
\end{proof}

Next we prove Lemma \ref{lem:bias_of_entropy} regarding the entropy of the $k$-block empirical measure $\hat{\mu}_{k,n}$.
Our proof relies on the following well-known bound in information theory.
A proof may be found in \cite[Theorem 17.3.3]{cover2012elements}.

\begin{elem}\label{lem:entropy_lipschitz}
Let $\alpha$ and $\beta$ be two probability measures on a finite set $\calU$ such that $\|\alpha - \beta\|_1 \leq \frac{1}{2}$.
Then
\begin{equation*}
|H(\alpha) - H(\beta)| \leq \|\alpha - \beta\|_1 \log\left(\frac{|\calU|}{\|\alpha - \beta\|_1}\right).
\end{equation*}
\end{elem}

\biasofentropy*
\begin{proof}
Fix $k \geq 1$ and $n \geq k$ to be chosen later on, and let $\Delta_{k,n} = \|\mu_k - \hat{\mu}_{k,n}\|_1$.
For a fixed sequence $X_1^n$, Lemma \ref{lem:entropy_lipschitz} gives us
\begin{align*}
|H(\mu_k) - H(\hat{\mu}_{k,n})| &\leq \Delta_{k,n} \log\left(\frac{|\mathcal{X}|^k}{\Delta_{k,n}}\right) \1(\Delta_{k,n} \leq \nicefrac{1}{2}) + (\log |\mathcal{X}|^k) \1(\Delta_{k,n} > \nicefrac{1}{2}) \\
&\leq \Delta_{k,n} \log\left(\frac{|\mathcal{X}|^k}{\Delta_{k,n}}\right) + (\log |\mathcal{X}|^k) \1(\Delta_{k,n} > \nicefrac{1}{2}).
\end{align*}
Taking an expectation, we have
\begin{align*}
\mathbb{E}|H(\mu_k) - H(\hat{\mu}_{k,n})| &\leq \mathbb{E}\left[\Delta_{k,n} \log\left(\frac{|\mathcal{X}|^k}{\Delta_{k,n}}\right)\right] + \mathbb{P}(\Delta_{k,n} > \nicefrac{1}{2}) \log |\mathcal{X}|^k \\
&\leq \mathbb{E}\left[\Delta_{k,n} \log\left(\frac{|\mathcal{X}|^k}{\Delta_{k,n}}\right)\right] + 2 (\mathbb{E}\Delta_{k,n}) \log |\mathcal{X}|^k \\
&\leq \mathbb{E}\left[\Delta_{k,n} \log\left(\frac{|\mathcal{X}|^k}{\Delta_{k,n}}\right)\right] +  u(k,n) \log |\mathcal{X}|^{2k},
\end{align*}
for $n$ large enough.
An application of Jensen's inequality to the first term yields
\begin{equation*}
\mathbb{E}\left[\Delta_{k,n} \log\left(\frac{|\mathcal{X}|^k}{\Delta_{k,n}}\right)\right] \leq (\mathbb{E} \Delta_{k,n}) \log\left(\frac{|\mathcal{X}|^k}{\mathbb{E} \Delta_{k,n}}\right).
\end{equation*}
Since $\mathbb{E}\Delta_{k,n} \leq u(k,n)$ and $\lim_{n \rightarrow\infty} u(k,n) = 0$, for large enough $n$, $\mathbb{E}\Delta_{k,n}$ and $u(k,n)$ are both less than or equal to $\nicefrac{1}{2}$.
Since $x \mapsto x\log(\nicefrac{|\mathcal{X}|^k}{x})$ is increasing on $x \in [0, \nicefrac{1}{2}]$, this implies that
\begin{equation*}
\mathbb{E}\left[\Delta_{k,n} \log\left(\frac{|\mathcal{X}|^k}{\Delta_{k,n}}\right)\right] \leq u(k,n) \log\left(\frac{|\mathcal{X}|^k}{u(k,n)}\right).
\end{equation*}
Combining this with our earlier bound and gathering terms, we obtain the result.
\end{proof}

\section{Properties of the Proposed Estimates}\label{app:estimate_properties}
\begin{prop}
\label{prop:entropic_est_is_joining}
For any $\eta \geq 0$, the proposed estimates satisfy $\esteojkn \in \J(\Lambda^k[\hat{\mu}_{k,n}], \Lambda^k[\hat{\nu}_{k,n}])$ and 
\begin{equation*}
\int c \, d\esteojkn_1 - \eta h(\esteojkn) = \eojhatkn.
\end{equation*}
\end{prop}
\begin{proof}
Let $\pi \in \calM(\X^k \times \Y^k)$ be as defined in Section \ref{sec:entropic_oj}.
We start by proving that $\esteojkn$ is invariant under the left-shift $\sigma \times \tau$ on $\X^\bbN \times \Y^\bbN$.
Note first that by construction, $\tilde{\Lambda}^k[\pi] \circ (\sigma \times \tau)^{-k} = \tilde{\Lambda}^k[\pi]$.
Then,
\begin{align*}
k \esteojkn \circ (\sigma \times \tau)^{-1} &= \sum\limits_{\ell = 0}^{k-1} \tilde{\Lambda}^k[\pi] \circ (\sigma \times \tau)^{-\ell - 1} \\
&= \sum\limits_{\ell = 1}^{k-1} \tilde{\Lambda}^k[\pi] \circ (\sigma \times \tau)^{-\ell} + \tilde{\Lambda}^k[\pi] \circ (\sigma \times \tau)^{-k} \\
&= \sum\limits_{\ell = 1}^{k-1} \tilde{\Lambda}^k[\pi] \circ (\sigma \times \tau)^{-\ell} + \tilde{\Lambda}^k[\pi] \\
&= \sum\limits_{\ell = 0}^{k-1} \tilde{\Lambda}^k[\pi] \circ (\sigma \times \tau)^{-\ell} \\
&= k \esteojkn.
\end{align*}

\noindent Thus $\esteojkn \in \calM_s(\X^\bbN \times \Y^\bbN)$.
Next we prove that $\esteojkn \in \Pi(\Lambda^k[\hat{\mu}_{k,n}], \Lambda^k[\hat{\nu}_{k,n}])$.
Fix a measurable set $C \subset \X^\bbN$.
Then 
\begin{align*}
\esteojkn(C \times \Y^\bbN) &= \frac{1}{k} \sum\limits_{\ell=0}^{k-1} \tilde{\Lambda}^k[\pi](\sigma^{-\ell} C \times \tau^{-\ell} \Y^\bbN) \\
&= \frac{1}{k} \sum\limits_{\ell=0}^{k-1} \tilde{\Lambda}^k[\pi](\sigma^{-\ell} C \times \Y^\bbN) \\
&= \frac{1}{k} \sum\limits_{\ell=0}^{k-1} \tilde{\Lambda}^k[\hat{\mu}_{k,n}](\sigma^{-\ell} C) \\
&= \Lambda^k[\hat{\mu}_{k,n}](C).
\end{align*}

\noindent Since $C$ was arbitary, it follows that the $\X^\bbN$-marginal of $\esteojkn$ is $\Lambda^k[\hat{\mu}_{k,n}]$.
A similar argument will show that the $\Y^\bbN$-marginal of $\esteojkn$ is $\Lambda^k[\hat{\nu}_{k,n}]$.
Thus $\esteojkn \in \J(\Lambda^k[\hat{\mu}_{k,n}], \Lambda^k[\hat{\nu}_{k,n}])$.
Finally, by the construction of $\esteojkn$,
\begin{equation*}
\int c \, d\esteojkn_1 - \eta h(\esteojkn) = \frac{1}{k}\int c_k \, d\pi - \frac{\eta}{k} H(\pi) = \frac{1}{k}\eotck\left(\hat{\mu}_{k,n}, \hat{\nu}_{k,n}\right) = \eojhatkn,
\end{equation*}

\noindent where the first equality follows Lemma \ref{lem:invariance_of_entropy_rate}.
\end{proof}

\section{Existence of an Entropic Optimal Joining}\label{app:existence}
\begin{prop}\label{prop:existence_of_entropic_oj}
Let $\X$ and $\Y$ be finite, $\mu \in \calM_s(\X^\bbN)$, $\nu \in \calM_s(\Y^\bbN)$, and $c: \X \times \Y \rightarrow \bbR_+$ be a non-negative cost function.
Then for every $\eta \geq 0$, the set of entropic optimal joinings $\Jmin^\eta(\mu, \nu)$ is non-empty.
\end{prop}
\begin{proof}
Fix $\eta \geq 0$ and a sequence $\{\lambda^n\} \subset \J(\mu, \nu)$ such that $\int c \, d\lambda_1^n - \eta h(\lambda^n) \rightarrow \eoj(\mu, \nu)$.
As $\J(\mu, \nu)$ is compact in the weak topology, we may extract a subsequence $\{\lambda^{n_\ell}\}$ such that $\lambda^{n_\ell} \Rightarrow \lambda$ as $\ell \rightarrow\infty$ for some $\lambda \in \J(\mu, \nu)$.
As the entropy rate $h(\cdot)$ is weakly upper semicontinuous on $\calM_s(\X^\bbN \times \Y^\bbN)$ and $c$ is continuous and bounded,
\begin{equation*}
\int c \, d\lambda_1 - \eta h(\lambda) \leq \liminf\limits_{\ell\rightarrow\infty} \left\{\int c \, d\lambda^{n_\ell}_1 - \eta h(\lambda^{n_\ell})\right\}= \eoj(\mu, \nu).
\end{equation*}

\noindent Thus we conclude that $\lambda \in \Jmin^\eta(\mu, \nu)$ and $\Jmin^\eta(\mu, \nu)$ is non-empty.
\end{proof}

\section{Properties of the $(c, \eta)$-Transform}
\label{sec:transform_properties}
\cetatransformfacts*
\begin{proof}
We will prove the bound for $\tg$ and the bound for $\tf$ will follow from a similar argument.
Let the conditions of the proposition hold.
Then for any $u, u' \in \calU$,
\begin{align*}
\tg(u) &= - \eta \log \left(\sum\limits_v \exp\left\{\frac{1}{\eta} (g(v) - c(u,v))\right\}\right) \\
&= -\eta \log \left(\sum\limits_v \exp\left\{\frac{1}{\eta} (g(v) - c(u',v) + c(u', v) - c(u,v))\right\}\right) \\
&= -\eta \log \left( \sum\limits_v \exp\left\{\frac{1}{\eta} (g(v) - c(u',v))\right\} \exp\left\{\frac{1}{\eta}(c(u', v) - c(u,v))\right\}\right) \\
&\leq -\eta \log \left( \exp\left\{-\frac{M}{\eta} d_\calU(u, u')\right\} \sum\limits_v \exp\left\{\frac{1}{\eta} (g(v) - c(u',v))\right\}\right) \\
&= M d_\calU(u, u') + \tg(u').
\end{align*}
\noindent Applying the same argument after exchanging $u$ and $u'$ and using the symmetry of $d_\calU(\cdot, \cdot)$, the result for $\tg$ follows.
\end{proof}

\section{Weak Convergence of Couplings and Joinings}
\label{app:weakconvergence}
\begin{lem}\label{lemma:weakly_convergent_subsequence}
Let $\calU$ and $\calV$ be Polish spaces and $\{\mu^n\} \subset \calM(\calU)$ and $\{\nu^n\} \subset \calM(\calV)$ be sequences satisfying $\mu^n \Rightarrow \mu$ and $\nu^n \Rightarrow \nu$ for some $\mu \in \calM(\calU)$ and $\nu \in \calM(\calV)$.
Then for any sequence $\{\pi^n\}$ satisfying $\pi^n \in \Pi(\mu^n, \nu^n)$ for every $n \geq 1$, there exists a subsequence $\{\pi^{n_\ell}\}$ such that $\pi^{n_\ell} \Rightarrow \pi$ for some $\pi \in \Pi(\mu, \nu)$.
Moreover, if $\calU = \X^\bbN$ and $\calV = \Y^\bbN$ for Polish alphabets $\X$ and $\Y$ and for every $n \geq 1$, $\mu^n \in \calM_s(\X^\bbN)$, $\nu^n \in \calM_s(\Y^\bbN)$ and $\pi^n \in \J(\mu^n, \nu^n)$, then $\pi \in \J(\mu, \nu)$.
\end{lem}
\begin{proof}
The first part of the lemma follows from basic weak convergence arguments (see for example \cite{villani2008optimal}).
Suppose that the second set of conditions hold.
Then from the first part of the lemma, it suffices to show that $\pi$ is invariant under the joint left-shift $\sigma \times \tau: \X^\bbN \times \Y^\bbN \rightarrow \X^\bbN \times \Y^\bbN$.
Since $\sigma \times \tau$ is continuous, for any bounded and continuous $f: \X^\bbN \times \Y^\bbN \rightarrow \mathbb{R}$, $f \circ (\sigma \times \tau)$ is also bounded and continuous and it follows that
\begin{equation*}
\int f \, d[\pi \circ (\sigma \times \tau)^{-1}] = \int f \circ (\sigma \times \tau) \, d\pi = \lim\limits_{\ell \rightarrow\infty} \int f \circ (\sigma \times \tau) \, d\pi^{n_\ell} = \lim\limits_{\ell \rightarrow\infty} \int f \, d\pi^{n_\ell} = \int f \, d\pi.
\end{equation*}

\noindent Since $f$ was arbitrary, we conclude that $\pi$ is invariant under the left-shift $\sigma\times\tau$ and the second part of the lemma follows.
\end{proof}

\end{document}